\documentclass[11pt]{article}

\usepackage[T1]{fontenc}
\usepackage[utf8]{inputenc}
\usepackage{lmodern}

\usepackage[a4paper,margin=1in]{geometry}
\usepackage{amsmath,amssymb,amsthm,mathtools}
\usepackage{enumitem}
\usepackage{mathrsfs}
\usepackage{hyperref}
\usepackage{stmaryrd}
\usepackage{xcolor}
\usepackage{tikz}

\theoremstyle{plain}
\newtheorem{theorem}{Theorem}[section]
\newtheorem{proposition}[theorem]{Proposition}
\newtheorem{lemma}[theorem]{Lemma}
\newtheorem{corollary}[theorem]{Corollary}

\theoremstyle{definition}

\theoremstyle{remark}
\newtheorem{remark}[theorem]{Remark}

\newcounter{proofstep}

\title{Blaschke--Santal{\'o} diagram for the volume, the diameter, and the Cheeger constant}
\author{Zakaria Fattah\footnote{Laboratoire de Math\'ematiques, Informatiques, Physique et Applications. N\^{\i}mes Universit\'e, Site des Carmes, 7 Place Gabriel P\'eri, 30000 N\^{\i}mes, France. Email address: \texttt{zakaria.fattah@unimes.fr}}}
\date{\today}

\begin{document}

\maketitle

\begin{abstract}
We study the Blaschke--Santal{\'o} diagram associated with the volume, the diameter, and the Cheeger constant. In the class of bounded open subsets of $\mathbb{R}^m$, $m\ge 2$, we give a complete description of the diagram. In the class of convex bodies of $\mathbb R^m$, we prove that the diagram is closed and simply connected as it is given by the region between two continuous functions for which qualitative properties such as  monotonicity, local behavior, and growth estimates are studied.   
\end{abstract}

% \tableofcontents
\section{Introduction and main results}

Let \(m\ge2\), and let \(\Omega\subset\mathbb R^m\) be a bounded open set. The
Cheeger constant of \(\Omega\) is defined by
\begin{equation}\label{eq:cheeger}
    h(\Omega):=
\inf\left\{
\frac{P(E)}{|E|}:\ E\subset\Omega,\ E\text{ has finite perimeter},\ |E|>0
\right\},
\end{equation}
where \(P(E)\) denotes the De Giorgi perimeter. This quantity has been first introduced by Jeff Cheeger to provide a lower bound for the first eigenvalue of the
Laplacian \cite{Cheeger1970}. It has since become a
classical functional in geometric analysis and shape optimization, see \cite{Leonardi2015} for a review on the Cheeger problem. 

Any admissible set \(E\subset\Omega\) realizing the infimum \eqref{eq:cheeger} is called a Cheeger set of \(\Omega\). For convex domains, it has been proved that the corresponding Cheeger sets are unique and convex, see  \cite{AlterCaselles2009,CasellesChambolleNovaga2010}, while for nonconvex domains, questions of existence, uniqueness, and regularity are more closely related to the geometry of the domain, see, for example, \cite{Leonardi2015,LeonardiNeumayerSaracco2017,LeonardiPratelli2016,parini_review}. These results give qualitative information on Cheeger sets, but they do not provide general constructive results.

In \cite{KawohlLachandRobert2006}, Kawohl and Lachand-Robert  provide an explicit description of the Cheeger set of planar convex bodies. However, no similar result is known for arbitrary convex bodies in higher dimensions. Consequently, in dimensions \(m\ge3\), it is unlikely to find explicit values of the Cheeger constant. We are only aware of the result of \cite{parini_bobkov},where the authors provide an explicit description of Cheeger sets for rationally invariant domains. 

Motivated by obtaining theoretical estimates of the Cheeger constant, it is natural to study sharp inequalities relating it to more elementary geometric quantities that can be explicitly computed. This leads to consider Blaschke--Santal\'o diagrams that are  tools allowing to describe the existing sharp inequalities relating given shape functionals defined on a some class of sets. Such diagrams have been introduced for the first time in \cite{Blaschke1915} by Blaschke who studied extremal problems for planar convex bodies, in particular minimizing the area among convex sets of prescribed constant width. Then, came the work of Santal\'o \cite{Santalo1961} who presented different complete systems of inequalities involving purely geometric functionals, namely, the area, the perimeter, the inradius, the diameter, the circumradius and the minimal width.  

In recent years, Blaschke--Santal\'o diagrams have been studied for various triplets of shape functionals in particular for the class of convex bodies. When dealing with purely geometric functionals, one is interested in finding explicit descriptions of the diagrams; we refer to the following non-exhaustive list of papers \cite{cifre3,MR3653891,brand,DelyonHenrotPrivat2022,cifre4,cifre}, where several Blaschke--Santal\'o diagrams have been fully characterized. We note that there are still some diagrams involving purely geometric functionals on planar convex sets that remain open. We refer to the discussion in the introduction of \cite{DelyonHenrotPrivat2022} for a state of the art of the still missing geometric diagrams. On the other hand, several authors have investigated diagrams involving combinations of geometric and spectral functionals. Since an explicit characterization of such diagrams is generally difficult, and in some cases even impossible to obtain, the focus in this setting is instead on establishing qualitative properties of the diagrams. We refer to \cite{freitas_antunes,MR3068840,zbMATH01308639,ftouhi_henrot,FtouhiLamboley2021,LZ,zbMATH06736468} for a non-exhaustive list of works in this context. At last, we note that Blaschke--Santal\'o diagrams have also been investigated from a numerical perspective; we refer to \cite{BogoselButtazzoOudet2023,Ftouhi2025Numerical,MartinetFtouhi2026} for various numerical approaches.  

More recently, there has been an interest in the study of Blaschke--Santal\'o diagrams relating the Cheeger constant to simple geometric quantities. We refer to \cite{ftouhi_inequality,ftouhi_cheeger,FtouhiMasielloPaoli2024} for complete descriptions of some relevant diagrams in the class of planar convex bodies and interesting conjectures. 

In the present work, we focus on the study of diagrams involving the Cheeger constant, the diameter and the volume in any dimension $m$ and for two classes of sets: \(\mathcal O^m\) the class of nonempty bounded open subsets of
\(\mathbb R^m\), and \(\mathcal K^m\) the class of convex bodies in
\(\mathbb R^m\). More precisely, we consider the following two diagrams: 
\begin{equation}\label{eq:def-convex-diagram}
\mathring{\mathcal D}_m
:=
\left\{
\bigl(D(\Omega),h(\Omega)\bigr):
\Omega\in\mathcal O_1^m
\right\},
\qquad
\mathcal D_m
:=
\left\{
\bigl(D(K),h(K)\bigr):
K\in\mathcal K_1^m
\right\},
\end{equation}
where \(|\cdot|\) denotes the \(m\)-dimensional
Lebesgue measure, \(D(E):=\sup\{|x-y|:\ x,y\in E\}\) the diameter of a bounded set \(E\subset\mathbb R^m\) and $\mathcal O_1^m:=\{\Omega\in\mathcal O^m:|\Omega|=1\},
$ and $
\mathcal K_1^m:=\{K\in\mathcal K^m:|K|=1\}$.

We begin by stating the result obtained for the class of bounded open sets, where we have obtained a complete description of the corresponding diagram.
\begin{theorem}[The Blaschke--Santal\'o diagram in the class of bounded open sets]
\label{thm:main-open-diagram}
Let \(B\subset\mathbb R^m\) be the Euclidean ball of volume \(1\). Then
\[
\mathring{\mathcal D}_m
=
\bigl(D(B),+\infty\bigr)\times\bigl(h(B),+\infty\bigr)
\cup
\{(D(B),h(B))\}.
\]
\end{theorem}
The result of this theorem can be read as the fact that in the class of bounded open sets, the diagram is completely determined by two sharp estimates: the
isodiametric inequality i.e., $D(\Omega)/|\Omega|^{\frac{1}{m}}\ge D(B)/|B|^{\frac{1}{m}}$ and the Faber--Krahn inequality for the Cheeger constant i.e., $|\Omega|^{1/m}h(\Omega) \ge |B|^{1/m}h(B).$, where $B$ is a ball. To prove this result, we show that every pair $(x,y)\in (D(B),+\infty)\times(h(B),+\infty)$
is realized by a bounded open set of volume \(1\). The situation is
different under the convexity constraint: the admissible set is no longer a quadrant, and its description is given by the
following theorem.
\begin{theorem}[The Blaschke--Santal\'o diagram in the class of convex bodies]
\label{thm:main-convex-bodies-diagram}
Let \(B\subset\mathbb R^m\) be the Euclidean ball of volume \(1\), and set
\(d_0:=D(B)\). For every \(x\ge d_0\), define
\begin{equation}\label{eq:def-boundary-functions}
f(x):=\inf\bigl\{h(K):K\in\mathcal K_1^m,\ D(K)=x\bigr\},
\qquad
g(x):=\sup\bigl\{h(K):K\in\mathcal K_1^m,\ D(K)=x\bigr\}.
\end{equation}
Then the following assertions hold.
\begin{enumerate}[label=\textnormal{(\roman*)}]
\item The diagram \(\mathcal D_m\) is the region between the graphs of \(f\)
and \(g\), namely
\[
\mathcal D_m
=
\bigl\{(x,y)\in\mathbb R^2:\ x\ge d_0,\ f(x)\le y\le g(x)\bigr\}.
\]

\item The functions \(f\) and \(g\) are continuous on \([d_0,+\infty)\).

\item The function \(f\) is strictly increasing on \([d_0,+\infty)\).

\item If \(m=2\), then \(g\) is strictly increasing on \([d_0,+\infty)\).

\item There exist constants \(c_m,C_m>0\), depending only on \(m\), such that,
for every \(x>d_0\),
\[
c_m x^{1/(m-1)}
\le f(x)\le g(x)\le C_m x^{m-1}.
\]
Moreover, the exponents \(1/(m-1)\) and \(m-1\) are sharp. 

\item The lower boundary function \(f\) has a right derivative at \(d_0\), and
$
f'_+(d_0)=0.$
\end{enumerate}
\end{theorem}

This theorem provides a global description of the diagram in the class of convex bodies. The first
assertion states that the diagram is a simply connected region of $\mathbb R^2$. The following assertions describe the continuity and monotonicity of these boundary functions. Indeed, it is shown in any dimension that the lower boundary is given  by the graph of a continuous and strictly increasing in every dimension. The strict monotonicity of the upper boundary is only proved in dimension \(2\). A generalization of this result seems out of reach for the moment due to the lack of information on the Cheeger problem in higher dimensions $m\ge 3$. The growth estimates provide the sharp powers of the growth of the boundary functions $f$ and $g$ when $x\to +\infty$, and the identity
$f'_+(d_0)=0$ describes the first-order behavior of the lower boundary near the point $(D(B),h(B))$ corresponding to the Euclidean ball.

The paper is organized as follows. Section \ref{sec:proof_open} is devoted to the proof of Theorem~\ref{thm:main-open-diagram}. In Section \ref{sec:notations}, we introduce some notations and notions from convex geometry, and present some important lemmas used in the sequel. Finally, in the involved Section \ref{sec:thm_conv}, we prove Theorem~\ref{thm:main-convex-bodies-diagram} and the corollaries describing
the geometry of the convex diagram.

%========================================================
\section{Proof of Theorem \ref{thm:main-open-diagram}: the diagram of open sets}\label{sec:proof_open}
%========================================================
We now prove Theorem~\ref{thm:main-open-diagram}. Recall that
\[
\mathring{\mathcal D}_m
:=
\left\{
\bigl(D(\Omega),h(\Omega)\bigr):
\Omega\in\mathcal O_1^m
\right\},
\qquad
\mathcal O_1^m:=\{\Omega\in\mathcal O^m:|\Omega|=1\}.
\]
Let \(B\subset\mathbb R^m\) be the Euclidean ball of volume \(1\). We prove that
\[
\mathring{\mathcal D}_m
=
\bigl(D(B),+\infty\bigr)\times\bigl(h(B),+\infty\bigr)
\;\cup\;
\{(D(B),h(B))\}.
\]
\begin{proof}
%Let \(R_0:=\omega_m^{-1/m}\), so that the unit-volume ball \(B\) has radius \(R_0\), diameter \(D(B)=2R_0\), and Cheeger constant \(h(B)=m/R_0\).

Let us first show the direct inclusion. Consider \(\Omega\in\mathcal O_1\). By the isodiametric inequality, one has \(D(\Omega)\ge D(B)\), with equality if and only if \(\Omega\) is a ball of volume \(1\). Moreover, it is well known that the ball minimizes the Cheeger constant among sets of a given volume. Therefore \(h(\Omega)\ge h(B)\). Thus, every point \((D(\Omega),h(\Omega))\) belongs to \((D(B),+\infty)\times (h(B),+\infty)\cup\{(D(B),h(B))\}\), as claimed.

Let us now prove the reverse inclusion. Fix \(x>D(B)\) and \(y>h(B)\), and set \(r:=m/y\). Since \(B\) is the unit-volume ball, one has \(h(B)=2m/D(B)\). Hence, from \(y>h(B)\), we get
\[
r=\frac m y<\frac m{h(B)}=\frac{D(B)}2<\frac x2.
\]

Thus, \(B(0,r)\) has Cheeger constant \(y\) and volume strictly smaller than \(1\).

%We then enlarge this ball by adding a finite family of thin spherical shells, while keeping the Cheeger constant unchanged. 
Now, choose an integer \(N\ge 1\) so large that $N>(xy-2m)/2$ and define \(\Delta:=(x/2-r)/N\) and \(s_i:=r+i\Delta\)  for \(i\in \llbracket 1,N \rrbracket\). For \(t\in(0,\Delta]\), let us set
\[
A_i(t):=\{z\in\mathbb R^m:\ s_i-t<|z|<s_i\}, \qquad i\in \llbracket 1,N \rrbracket,
\]
and
\[
\Omega^t:=B(0,r)\cup \bigcup_{i=1}^N A_i(t).
\]

For \(0<t<\Delta\), the connected components of \(\Omega^t\) are pairwise disjoint, and the outermost shell reaches radius \( x/2\); hence \(D(\Omega^t)=x\) for every \(t\in(0,\Delta)\).
We choose \(t^*\in(0,\Delta)\) such that \(|\Omega^{t^*}|=1\).

It remains to compute the Cheeger constant of \(\Omega^{t^*}\). Since its connected components are disjoint and separated by a positive distance, \(h(\Omega^{t^*})\) is the minimum of the Cheeger constants of its connected components. It is well known that annuli are self-Cheeger sets, see for instance \cite[Theorem 2]{KrejcirikLeonardiVlachopulos}. Therefore, for every
\(i\in[\![1,N]\!]\), we have
\[
h(A_i(t))
=
m\,\frac{s_i^{m-1}+(s_i-t)^{m-1}}{s_i^m-(s_i-t)^m}.
\]

By the mean value theorem,
\(
s_i^m-(s_i-t)^m\leq m t s_i^{m-1}.
\)
Therefore,
\[
h(A_i(t))
=
m\,\frac{s_i^{m-1}+(s_i-t)^{m-1}}
{s_i^m-(s_i-t)^m}
\geq
\frac1t
\geq
\frac1\Delta .
\]

Since \(r=m/y\), the condition \(N>(xy-2m)/2\) is equivalent to
\(
\Delta=(x/2-r)/N<1/y
\). Hence, for every \(i=1,\ldots,N\) and every \(t\in(0,\Delta)\), we have
\(
h(A_i(t))\ge 1/\Delta>y.
\)
Therefore,
\[
h(\Omega^{t^*})
=
\min\Bigl(
h(B(0,r)),\, h(A_1(t^*)),\,\ldots,\,h(A_N(t^*))
\Bigr)
=
y.
\]

We have thus constructed an open set \(\Omega^{t^*}\in\mathcal O_1\) such that $(D(\Omega^{t^*}),h(\Omega^{t^*}))=(x,y)$, which proves the inclusion
\[
(D(B),+\infty)\times(h(B),+\infty)\subset \mathring{\mathcal D}_m.
\]

Finally, the point \((D(B),h(B))\) being realized by the unit-volume ball itself, the proof is complete.
\end{proof}

%========================================================
\section{Notations and important lemmas}\label{sec:notations}
%========================================================
In this section, we fix notation from convex geometry and prove some  quantitative estimates used in the sequel. We first derive a lower bound for the inradius in
terms of the volume and the diameter . We then prove a two-sided estimate for the
Cheeger constant in terms of volume and diameter; on the class \(\mathcal K_1^m\),
this gives the growth bounds for the boundary functions of the diagram. Finally,
we give a quantitative estimate for the Cheeger constant with respect to the
Hausdorff distance.
\subsection{Convex bodies and notation}\label{subsec:convex-bodies-notation}
%We begin with the Hausdorff topology on convex bodies and the notation forsupport functions, diameter, inradius, quermassintegrals, and mixed volumes.

If $n, m\in \mathbb{N}$ such that $n\leq m$, we denote by $\llbracket n,m \rrbracket$ the set of integers $\{n,n+1,\dots,m\}$.

A convex body in \(\mathbb R^m\) is a compact convex subset of
\(\mathbb R^m\) with nonempty interior. For \(K\in\mathcal K^m\), we set
$|K|:=\mathcal H^m(K),$ and $P(K):=\mathcal H^{m-1}(\partial K)$,
where \(\mathcal H^k\) denotes the \(k\)-dimensional Hausdorff measure.
Since \(K\) is convex, this notion of perimeter coincides with the
De Giorgi perimeter; see for instance \cite[Example~12.6, p.~124]{Maggi2012}.

We denote by \(\mathcal K^m\) the class of convex bodies in
\(\mathbb R^m\), and by
$
\mathcal K_1^m
:=
\{K\in\mathcal K^m:\ |K|=1\}
$
the corresponding volume-normalized class. Convex bodies will be considered
up to translations.

The support function of a convex body \(K\in\mathcal K^m\) is defined as follows
\[
h_K:u\in\mathbb S^{m-1}\longmapsto\sup_{x\in K}\langle x,u\rangle.
\]

It is well known that the Hausdorff distance on the class \(\mathcal K^m\) is given by
\[
d_H(K,L)
:=
\|h_K-h_L\|_\infty
=
\sup_{u\in\mathbb S^{m-1}}
|h_K(u)-h_L(u)|,
\]
see, for instance, \cite[Lemma 1.8.14]{Schneider2014}.

Thus, \(K_n\underset{n\rightarrow+\infty}{\longrightarrow} K\) in the Hausdorff metric if and only if
\(h_{K_n}\underset{n\rightarrow+\infty}{\longrightarrow} h_K\) uniformly on \(\mathbb S^{m-1}\).

We recall the definitions of the diameter and the inradius
\[
D(K):=\sup\{|x-y|:\ x,y\in K\},
\]
and
\[
r(K):=
\sup\{r>0:\ \exists x\in\mathbb R^m
\text{ such that } B(x,r)\subset K\}.
\]

%\textcolor{blue}{We also recall the quermassintegrals \(W_0(K),\dots,W_m(K)\), they appear in
%Steiner's formula} \ilias{Toute cette phrase est à refaire}
% \textcolor{blue}{\(\star\)}We also recall that the quermassintegrals \(W_0(K),\dots,W_m(K)\) are the
% coefficients in Steiner's formula
% \[
% |K+tB_1|
% =
% \sum_{i=0}^{m}
% \binom{m}{i}W_i(K)t^i,
% \qquad t\ge0,
% \]
% where \(B_1\) denotes the Euclidean unit ball. In particular,
% $W_0(K)=|K|,$ and $
% mW_1(K)=P(K)$.

% \underline{We shall} use mixed volumes only through Minkowski's polynomial formula. \ilias{cette phrase est bizarre}
% For \(K,L\in\mathcal K^m\), we write \(W_k(K,L)\), \(k=0,\dots,m\), for
% the coefficients determined by
% \begin{equation}\label{eq:minko}
%    |(1-t)K+tL|
% =
% \sum_{k=0}^{m}
% \binom{m}{k}
% (1-t)^{m-k}t^k
% W_k(K,L),
% \qquad t\in[0,1]. 
% \end{equation}
% \textcolor{blue}{\(\star\)}For \(K,L\in\mathcal K^m\), Minkowski's theorem gives 
% \begin{equation}\label{eq:minko}
%    |(1-t)K+tL|
% =
% \sum_{k=0}^{m}
% \binom{m}{k}
% (1-t)^{m-k}t^k
% W_k(K,L),
% \qquad t\in[0,1],
% \end{equation}
% where \(W_k(K,L)\), \(k=0,\dots,m\), are the corresponding mixed volumes.

% \textcolor{blue}{\(\star\)}The following continuity property will be used repeatedly: if \(K_n\to K\) and \(L_n\to L\) in the
% Hausdorff metric, then
% \[
% W_k(K_n,L_n)\underset{n\rightarrow +\infty}{\longrightarrow} W_k(K,L),
% \qquad k=0,\dots,m.
% \]
For \(K,L\in\mathcal K^m\), Minkowski's theorem, see for instance \cite[Theorem 5.1.7]{Schneider2014},  gives 
\begin{equation}\label{eq:minko}
 \forall t\in[0,1],\ \ \ \   |(1-t)K+tL|
=
\sum_{k=0}^{m}
\binom{m}{k}
(1-t)^{m-k}t^k
W_k(K,L),
\end{equation}
where the coefficients \((W_k(K,L))_{k\in \llbracket 0,m \rrbracket}\) are called the mixed volumes of the sets $K$ and $L$.

The particular case $L=B_1$, where $B_1$ is the unit ball is of intrinsic interest and leads to the definition of the so-called quermassintegrals of the body $K$ given by $W_k(K):=W_k(K,B_1)$ for $k\in \llbracket 0,m \rrbracket$. 
This allows to state the following special case known as Steiner's formula
\begin{equation}\label{eq:steiner}
\forall t\ge0,\ \ \ \     |K+tB_1|
=
\sum_{i=0}^{m}
\binom{m}{i}W_i(K)t^i.
\end{equation}
Note that in particular, \(W_0(K)=|K|\) and \(mW_1(K)=P(K)\).

The following classic continuity property, stated for instance in \cite[Page 280]{Schneider2014}, will be used repeatedly: if \(K_n\underset{n\to +\infty}{\longrightarrow} K\) and
\(L_n\underset{n\to +\infty}{\longrightarrow} L\) in the Hausdorff metric, then
\[
\forall k\in \llbracket 0,m \rrbracket,\ \ \ \ W_k(K_n,L_n)\underset{n\rightarrow +\infty}{\longrightarrow} W_k(K,L).
\]

% We now recall the notion of a tangential body. Let \(K\in\mathcal K^m\).
% A point \(x\in\partial K\) is called regular if there exists a unique supporting
% hyperplane to \(K\) at \(x\). We denote by \(U(K)\) the set of all outer unit
% normals to \(K\) at regular boundary points. The form body of \(K\) is defined by
% \[
% K^*
% :=
% \bigcap_{u\in U(K)}
% \{x\in\mathbb R^m:\langle x,u\rangle\le 1\}.
% \]

% The convex body \(K\) is called tangential if it is homothetic to \(K^*\).
% We refer to~\cite[Definition~1.1]{Ftouhi2025}.
Finally, it is classical that the volume, the quermassintegrals, and  the diameter are
continuous on \(\mathcal K^m\) with respect to the Hausdorff convergence.
For a standard background on these notions, we refer to \cite{Schneider2014}.

\subsection{Geometric estimates involving the diameter, the inradius, and the Cheeger constant}
% We next derive a lower bound for the inradius in terms of volume and
% diameter, prove two-sided diameter estimates for the Cheeger constant with
% optimal exponents, and establish a Lipschitz estimate with respect to the
% Hausdorff distance for monotone homogeneous functionals on convex bodies.
% The estimates in this subsection will be used in three places. The inradius
% bound turns a diameter bound, at fixed volume, into a uniform interior-ball
% condition, which is needed for the Hausdorff-continuity estimate for \(h\) \ilias{Tu veux dire quoi par the Hausdorff continuity estimate for $h$? Il faut être clair et ne pas laisser le lecteur galerer avec des nomenclatures floues}. The
% two-sided estimate for the Cheeger constant gives, on \(\mathcal K_1^m\), both
% a diameter bound at fixed Cheeger level \ilias{what is "fixed cheeger level" ?} and the growth bounds for the boundary
% functions \(f\) and \(g\) \ilias{where are such functions defined ?}. The last estimate gives a quantitative control of
% \(h\) with respect to the Hausdorff distance under a uniform lower bound on the
% inradius.

In this subsection, we establish several inequalities and estimates
needed in the sequel. 

% As the Cheeger constant is considered here on open sets, and since our analysis is carried out in the class of open convex sets, we shall use the notation
% \[
% h(K):=h(\operatorname{int}K)
% \]
% for every convex body \(K\). In particular, throughout the paper, a convex body and its interior will be identified without further comment.

Let us begin with a lower bound for the inradius in terms of volume and diameter.
\begin{lemma}\label{lem:lower_r_d_A}
    Let \(K\in \mathcal{K}^m\). We have 
   $$r(K)\ge \frac{2^{m-1}}{m\,\omega_m}\cdot\frac{|K|}{D(K)^{m-1}}.$$ 
\end{lemma}
\begin{proof}
    The inequality directly follows from combining the following Wills' inequality
    \[
r(K)\,P(K)\ge |K|+(m-1)\omega_m\,r(K)^m,
\]
see for instance \cite[\S 7.2, Eq.~(7.44)]{Schneider2014}, and the estimate 
$P(K)\le m\,\omega_m\Bigl(\frac{D(K)}{2}\Bigr)^{m-1},$
that can be found in \cite[p.~23, Eq.~(3)]{GritzmannWillsWrase1980}. 
\end{proof}

%A lower bound on the inradius in terms of the volume and the diameter follows by combining Wills' inequality with a classical perimeter--diameter estimate due to Kubota. More precisely, for every convex body \(K\subset\mathbb R^m\), Wills' inequality asserts that
%\[|K|-r(K)\,P(K)+(m-1)\omega_m\,r(K)^m\le 0,\]
%see \cite[\S 7.2, Eq.~(7.44)]{Schneider2014}. Since the last term is nonnegative, this gives \( |K|\le r(K)\,P(K) \), hence \( r(K)\ge |K|/P(K) \).

%To estimate \(P(K)\) in terms of \(D(K)\), we use Kubota's inequality,
%\[P(K)\le m\,\omega_m\Bigl(\frac{D(K)}{2}\Bigr)^{m-1},\]
%see Kubota~\cite{Kubota1930} and Gritzmann, Wills and Wrase~\cite[p.~23, Eq.~(3)]{GritzmannWillsWrase1980}. Combining the two estimates, we obtain
%\begin{equation}\label{eq:inradius-diameter}
%r(K)\ge \frac{2^{m-1}}{m\,\omega_m}\,\frac{|K|}{D(K)^{m-1}}.
%\end{equation}
%We shall use this estimate only to derive a lower bound on the inradius when the volume is fixed and the diameter is bounded above.

In the following lemma, we derive two-sided estimates for the Cheeger constant in terms of the volume and the diameter and prove the optimality of the exponents appearing in these estimates.

\begin{lemma}\label{lem:two-sided-cheeger-diameter}
Let \(K\in \mathcal{K}^m\). We have
\[
\left(\frac{\omega_{m-1}}{m}\right)^{1/(m-1)}
\left(\frac{D(K)}{|K|}\right)^{1/(m-1)}
\le
h(K)
\le
\frac{m\,\omega_m}{2^{m-1}|K|}\,D(K)^{m-1}.
\]
Moreover, the two exponents are optimal.
\end{lemma}

\begin{proof}

$\bullet$ The lower estimate follows from the inequality $h(K)\ge \frac1m\cdot\frac{P(K)}{|K|}$, presented in \cite[Corollary~5.2]{Brasco2018}
and the inequality $
P(K)^{m-1}
\ge
\omega_{m-1}\,D(K)\,\bigl(m|K|\bigr)^{m-2}$, see for instance \cite[p.~23, Eq.~(4)]{GritzmannWillsWrase1980}. 

$\bullet$ The upper bound is obtained by combining the following trivial estimate 
$h(K)\le \frac{P(K)}{|K|}$, with the inequality $
P(K)\le m\,\omega_m\left(\frac{D(K)}2\right)^{m-1}$,
see for instance \cite[p.~23, Eq.~(3)]{GritzmannWillsWrase1980}. 

$\bullet$ It remains to prove that the exponents in the term $D(K)$ are optimal. For \(L>0\), set
\(K_L:=L\cdot B^{m-1}_1\times [-1/2,1/2]\) the cylinder of basis the $(m-1)$-dimensional ball and altitude 1. We have 
$$|K_L|=\omega_{m-1}L^{m-1},\ D(K_L)=\sqrt{1+4L^2}\ \text{and}\  P(K_L)=2\omega_{m-1}L^{m-1}+(m-1)\omega_{m-1}L^{m-2}.$$ 
Hence, from
\(\frac{P(K)}{m|K|}\le h(K)\le \frac{P(K)}{|K|}\), we get
\[
\frac1m\left(2+\frac{m-1}{L}\right)\le h(K_L)\le
2+\frac{m-1}{L}.
\]

Let \(0<L\le1\). We have
\[
\frac{m-1}{mL}\le h(K_L)\le \frac{m+1}{L},
\]
whereas
\[
\left(\frac{D(K_L)}{|K_L|}\right)^{1/(m-1)}
=
\omega_{m-1}^{-1/(m-1)}L^{-1}
(1+4L^2)^{1/(2m-2)}.
\]
Since the last factor is bounded from above and from below by positive constants, there exist \(c,C>0\), independent of \(L\), such that
\[
\forall L\in (0,1],\ \ \ \  c\le
\frac{h(K_L)}
{\left(D(K_L)/|K_L|\right)^{1/(m-1)}}
\le C.
\]
This shows that the exponent \(1/(m-1)\) in the lower bound is optimal.

Using again \(\frac1m\left(2+\frac{m-1}{L}\right)\le h(K_L)\le 2+\frac{m-1}{L}, \) and since \(L\ge1\), we obtain 
\[
\frac{2}{m}\le h(K_L)\le m+1.
\]
Moreover,
\[
\frac{D(K_L)^{m-1}}{|K_L|}
=
\frac{(1+4L^2)^{(m-1)/2}}{\omega_{m-1}L^{m-1}}
\underset{L\rightarrow+\infty}{\longrightarrow}
\frac{2^{m-1}}{\omega_{m-1}}
.
\]
Therefore there exist \(c,C>0\), independent of \(L\), such that
\[
\forall L\ge1,\ \ \ \  c\le
\frac{h(K_L)}
{D(K_L)^{m-1}/|K_L|}
\le C.
\]
This proves the optimality of the exponent \(m-1\) in the upper bound.
\end{proof}

In the following lemma, we prove a quantitative estimate for monotonic homogeneous functionals in terms of the Hausdorff distance between two sets. 
\begin{lemma}\label{lem:cox-monotone-homogeneous}
Let $\mathcal J:\mathcal K^m\to \mathbb R$ be a positive functional which is decreasing with respect to set inclusion and homogeneous of degree $\alpha<0$. Let $K_1,K_2\in\mathcal K^m$ be such that, for some $\rho>0$,
$h_{K_1}(u),
h_{K_2}(u)\ge \rho$, for every $u\in\mathbb S^{m-1}$. Assume, moreover, that
$d_H(K_1,K_2)<\rho$. Then
\[
|\mathcal J(K_1)-\mathcal J(K_2)|
\le
C_\alpha\,\rho^{\alpha-1}\,d_H(K_1,K_2),
\]
where \(C_\alpha\) depends only on \(\alpha\).
\end{lemma}

\begin{proof}
We recall that the Hausdorff distance between convex bodies is given the uniform distance between their support functions, i.e.,
$d_H(K_1,K_2)=\|h_{K_1}-h_{K_2}\|_\infty$.

The lower bounds \(h_{K_1},h_{K_2}\ge \rho\) imply 
%\ilias{Je ne comprends pas cette phrase, de quelle lowerbound tu parle ? fais des references précises}
\[
K_1\subset
\left(1+\frac{d_H(K_1,K_2)}{\rho}\right)K_2
\ \ \text{and}\ \ 
K_2\subset
\left(1+\frac{d_H(K_1,K_2)}{\rho}\right)K_1.
\]
Indeed, for every $u\in\mathbb S^{m-1}$,
\[
h_{K_1}(u)
\le
h_{K_2}(u)+d_H(K_1,K_2)
\le
\left(1+\frac{d_H(K_1,K_2)}{\rho}\right)h_{K_2}(u),
\]
and the reverse inclusion follows by symmetry.

By the monotonicity and the homogeneity of $\mathcal J$, these inclusions yield
\[
\left(1+\frac{d_H(K_1,K_2)}{\rho}\right)^{\alpha}
\mathcal J(K_2)
\le
\mathcal J(K_1)
\le
\left(1+\frac{d_H(K_1,K_2)}{\rho}\right)^{-\alpha}
\mathcal J(K_2).
\]
Consequently,
\[
|\mathcal J(K_1)-\mathcal J(K_2)|
\le
\left[
\left(1+\frac{d_H(K_1,K_2)}{\rho}\right)^{-\alpha}-1
\right]\mathcal J(K_2).
\]

Since $d_H(K_1,K_2)<\rho$, the quantity $1+d_H(K_1,K_2)/\rho$
belongs to the interval $[1,2]$. Applying the mean value theorem to the function
$t\mapsto t^{-\alpha}$ on this interval gives
\[
\left(1+\frac{d_H(K_1,K_2)}{\rho}\right)^{-\alpha}-1
\le
c_\alpha\,\frac{d_H(K_1,K_2)}{\rho},
\]
where
$
c_\alpha
=
\sup_{1\le t\le 2}(-\alpha)t^{-\alpha-1}$.

It is straightforward to check that
\[
c_\alpha
=
\begin{cases}
-\alpha, & \text{if}\ \ -1\le \alpha<0,\\[2mm]
(-\alpha)\,2^{-\alpha-1}, & \text{if}\ \  \alpha<-1.
\end{cases}
\]

Finally, from $h_{K_2}\ge \rho$, we obtain the inclusion $B(0,\rho)\subset K_2$. Since $\mathcal J$ is decreasing with respect to inclusion, we have
$
\mathcal J(K_2)\le \mathcal J(B(0,\rho))$. Finally, combining the previous estimates and using the homogeneity of $\mathcal{J}$, and letting $C_\alpha:=c_\alpha\,\mathcal J(B_1)$, we obtain 
\[
|\mathcal J(K_1)-\mathcal J(K_2)|
\le
c_\alpha\,\mathcal J(B_1)\,\rho^{\alpha-1}\,d_H(K_1,K_2)=C_\alpha \rho^{\alpha-1} d_H(K_1,K_2).
\]

\end{proof}

\begin{remark}
We point out a refinement of Lemma~\ref{lem:cox-monotone-homogeneous} and its
application to the Cheeger constant.

\begin{enumerate}
\item In the case \(\alpha=-1\), the assumption $d_H(K_1,K_2)<\rho$
is not needed. Indeed, in this case the map \(t\mapsto t^{-\alpha}\) is linear, and the mean value estimate used in the proof is exact.

\item Applying Lemma~\ref{lem:cox-monotone-homogeneous} to the Cheeger constant, which is decreasing with respect to set inclusion and homogeneous of degree \(-1\), gives
\[
|h(K_1)-h(K_2)|
\le
\frac{m}{\rho^{2}}\,d_H(K_1,K_2).
\]
\end{enumerate}
\end{remark}

%========================================================
\subsection{Continuous Steiner symmetrization and perturbation results}
%\subsection{Symmetrization and local extremal properties} 
%\ilias{J'aime pas beaucoup ce titre "local extremal properties"}
%========================================================

% In this subsection, we recall \ilias{je dirais plutôt "we recall"} continuous Steiner symmetrization for convex
% bodies and use it to prove """a strict decrease of the Cheeger constant in a
% suitable direction""" \ilias{that one can find a local perturbation that strictly decreases the Cheeger constant of convex body while preserving its convexity and volume}. We then establish the """local extremal properties""" \ilias{ce terme n'a pas de sense} of the
% diameter and the Cheeger constant on \(\mathcal K_1^m\), which will be used
% in the study of the boundary functions \ilias{what are "boundary functions"? N'oublies pas que tu n'écris pas pour toi-même, tu écris pour des lecteurs qui ne connaissent pas ton travail. Il faut donc les aider en gardant une façon la plus générale de présenter. Donc à ta place, moi je dirai simplement: ""the study of the boundary of the diragram D, defined in ..., et je mets une reference à la definition du diagramme "" }.

%\subsection{Continuous Steiner symmetrization}
In this subsection, we recall continuous Steiner symmetrization for convex
bodies and use it to prove that one can find a local perturbation that strictly
decreases the Cheeger constant of a convex body while preserving its convexity
and volume. We then establish the corresponding perturbation properties of the
diameter and the Cheeger constant on \(\mathcal K_1^m\), which will be used
in the study of the boundary of the diagram \(\mathcal D_m\), defined in
\eqref{eq:def-convex-diagram}.

We begin by recalling the definition of continuous Steiner symmetrization and
the properties used below: preservation of convexity and volume, monotonicity
with respect to set inclusion, and continuity in the Hausdorff metric.
Fix \(u\in\mathbb S^{m-1}\), and let
$H(u):=\{x\in\mathbb R^m:\langle x,u\rangle=0\}$.

For \(K\in\mathcal K^m\), we denote by \(\operatorname{Proj}_u(K)\) the orthogonal projection of \(K\) onto \(H(u)\). For every \(x\in \operatorname{Proj}_u(K)\), the section of \(K\) along the line \(x+\mathbb Ru\) is a compact interval, which we write as
$
K_x=[a(x),b(x)]$.

The Steiner symmetrized of \(K\) with respect to \(H(u)\) is then given by
\[
S_u(K)
=
\Bigl\{
x+su:\ x\in \operatorname{Proj}_u(K),\ |s|\le \frac{b(x)-a(x)}{2}
\Bigr\}.
\]

For \(t\in[0,1]\), we define
\[
S_u^t(K)
=
\Bigl\{
x+su:\ x\in \operatorname{Proj}_u(K),\
(1-t)a(x)+\frac t2\bigl(a(x)-b(x)\bigr)\le s\le
(1-t)b(x)+\frac t2\bigl(b(x)-a(x)\bigr)
\Bigr\}.
\]
Thus \(S_u^0(K)=K\) and \(S_u^1(K)=S_u(K)\).

The family \(\bigl(S_u^t\bigr)_{t\in[0,1]}\) corresponds to Brock's continuous Steiner symmetrization, see \cite{Brock95}, up to a reparametrization of the time variable. We shall use the following standard properties: it preserves the convexity and the volume, it is monotone with respect to set inclusion, and it is continuous with respect to the Hausdorff metric. In particular, if \(K\subset L\), then
\[
S_u^t(K)\subset S_u^t(L)
\qquad\text{for every }t\in[0,1].
\]

We refer to Brock \cite[Theorem~1, relation~(8), and Theorem~4]{Brock95}. In the sequel, we write
\(
K_t:=S_u^t(K).
\)
After a rotation sending \(u\) to \(e_m\), we shall use the representation
\[
K=\{(x',y):x'\in A,\ a(x')<y<b(x')\},
\]
where \(A\subset\mathbb R^{m-1}\) is convex.
%\subsection{Perimeter decay under continuous Steiner symmetrization}

Before stating an important perturbation lemma for the Cheeger constant, let us prove an intermediate result for the perimeter that will be used in the sequel.
\begin{lemma}[Strict decay of the perimeter under continuous Steiner symmetrization for \(C^1\) bodies]
\label{lem:css-perimeter-decay}
Let \(K\in\mathcal K^m\) be a convex body with \(C^1\) boundary, and let \(u\in\mathbb S^{m-1}\). Assume that \(K\) is not symmetric with respect to the hyperplane orthogonal to \(u\). For \(t\in[0,1]\), set
\(
K_t:=S_u^t(K).
\)
Then
\[
P(K_t)<P(K)
\qquad\text{for every }t\in(0,1].
\]

\end{lemma}

\begin{proof}
We may assume  without loss of generality that \(u=e_m\). We write
\[
K=\{(x',y):x'\in A,\ a(x')<y<b(x')\},
\]
with \(A\subset\mathbb R^{m-1}\) convex and \(a,b\in C^1(A)\) because the boundary of \(K\) is \(C^1\). For \(t\in[0,1]\), the continuous Steiner symmetrization is of the form
\[
K_t=\{(x',y):x'\in A,\ a_t(x')<y<b_t(x')\},
\]
where
$a_t=(1-t/2)a- (t/2)b$ and $
b_t=(1-t/2)b- (t/2)a$.

Hence \(b_t-a_t=b-a\). By the area formula applied to the upper and lower graphs, it remains to compare the corresponding graph terms.

Set \(F(z):=\sqrt{1+|z|^2}\). Since \(F\) is even and strictly convex, we have
\begin{equation}\label{eq:convexity-at}
F(\nabla a_t)
\le
\Bigl(1-\frac t2\Bigr)F(\nabla a)
+
\frac t2 F(\nabla b).
\end{equation}

\begin{equation}\label{eq:convexity-bt}
F(\nabla b_t)
\le
\Bigl(1-\frac t2\Bigr)F(\nabla b)
+
\frac t2 F(\nabla a).
\end{equation}
Summing,
\[
F(\nabla a_t)+F(\nabla b_t)\le F(\nabla a)+F(\nabla b)\quad\text{on }A.
\]
Integrating over \(A\), we obtain
\[
P(K_t)\le P(K)\qquad\text{for every }t\in[0,1].
\]

Assume now that \(P(K_t)=P(K)\) for some \(t\in(0,1]\). Then
$G:=\bigl(F(\nabla a)+F(\nabla b)\bigr)-\bigl(F(\nabla a_t)+F(\nabla b_t)\bigr)$
is a continuous nonnegative function on \(A\) with zero integral, hence \(G\equiv 0\) on \(A\). Since the inequality \(G\ge 0\) is obtained by summing \eqref{eq:convexity-at} and \eqref{eq:convexity-bt} that are inequalities holding in the same direction, the pointwise equality $G=0$ on $A$ forces pointwise equalities in each of \eqref{eq:convexity-at} and \eqref{eq:convexity-bt}. Since both coefficients \(1-\tfrac t2\) and \(\tfrac t2\) are strictly positive and \(F\) is strictly convex, this implies $
\nabla a=-\nabla b\quad\text{on }A$.

Therefore, \(\nabla(a+b)=0\) on \(A\), and since \(A\) is connected, \(a+b\) is constant. It follows that all sections of \(K\) parallel to \(u\) have the same midpoint, so \(K\) is symmetric with respect to a hyperplane parallel to \(H(u)\), hence, up to translation, with respect to \(H(u)\).

In particular, if \(K\) is not symmetric with respect to \(H(u)\), then
$P(K_t)<P(K)\ \text{for every }t\in(0,1].$
\end{proof}
%\subsection{Cheeger decay under continuous Steiner symmetrization}
%\ilias{"Cheeger decay" n'a pas de sense, tu parle de la constante de Cheeger plutot}

We apply Lemma~\ref{lem:css-perimeter-decay} to the Cheeger set of a convex
body and obtain a strict decrease of the Cheeger constant along a suitable
continuous Steiner symmetrization whenever the body is not a Euclidean ball.
\begin{lemma}\label{lem:css-cheeger-decay}
Let \(K\in\mathcal K_1^m\) be a convex body that is not a Euclidean ball. Then there exists a direction \(u\in\mathbb S^{m-1}\) such that, setting
$K_t:=S_u^t(K)$, we have
\[
h(K_t)<h(K)
\qquad\text{for every}\ t\in(0,1].
\]
\end{lemma}

\begin{proof}
Let \(C_K\subset K\) denote the Cheeger set of \(K\). Since \(K\) is convex, \(C_K\) is unique, convex, and of class \(C^{1,1}\); see, for instance, \cite{Leonardi2015}. Moreover, by \cite{FigalliMaggiPratelli2010}, a ball is the Cheeger set of a convex body if and only if it is a ball. Therefore, since \(K\) is not a Euclidean ball, its Cheeger set \(C_K\) is also not a ball.

Let \(u\in\mathbb S^{m-1}\) be such that \(C_K\) is not symmetric with respect to \(H(u)\). We claim that \(K\) is not symmetric with respect to \(H(u)\). Indeed, if \(K\) were symmetric with respect to \(H(u)\), then the reflection of \(C_K\) across \(H(u)\) would also be a Cheeger set of \(K\). By the uniqueness of the Cheeger set in convex bodies, this reflection must coincide with \(C_K\). Hence \(C_K\) would be symmetric with respect to \(H(u)\), contradicting the choice of \(u\).

For \(t\in[0,1]\), set \(K_t:=S_u^t(K)\) and \(C_t:=S_u^t(C_K)\). Since continuous Steiner symmetrization preserves convexity and volume, one has \(K_t\in\mathcal K_1^m\) for every \(t\in[0,1]\). Moreover, monotonicity with respect to inclusion gives \(C_t\subset K_t\) for every \(t\in[0,1]\).

Since \(C_K\) is not symmetric with respect to \(H(u)\), Lemma~\ref{lem:css-perimeter-decay} yields \(P(C_t)<P(C_K)\) for every \(t\in(0,1]\). On the other hand, continuous Steiner symmetrization preserves volume, so \(|C_t|=|C_K|\) for every \(t\in[0,1]\). Since \(C_t\subset K_t\), the set \(C_t\) is an admissible competitor for the Cheeger problem in \(K_t\). Therefore, for every \(t\in(0,1]\),
\[
h(K_t)\le \frac{P(C_t)}{|C_t|}<\frac{P(C_K)}{|C_K|}=h(K).
\]
This proves the claim.
\end{proof}

\subsection{Perturbation lemmas for the diameter and the Cheeger constant}
We combine the preceding symmetrization result with volume-normalized
perturbations to derive the local extremal properties of the diameter and the
Cheeger constant on \(\mathcal K_1^m\).
\begin{lemma}[Local extremal properties of \(D\) and \(h\)]\label{lem:local-extrema}
Let \(B\) be the Euclidean ball of volume \(1\). Then the following statements hold.
\begin{enumerate}
\item The ball \(B\) is the unique local minimizer of the diameter \(D\) on \(\mathcal K_1^m\).
\item The ball \(B\) is the unique local minimizer of the Cheeger constant \(h\) on \(\mathcal K_1^m\).
\item The diameter \(D\) admits no local maximizer on \(\mathcal K_1^m\).
\item In dimension \(2\), the Cheeger constant \(h\) admits no local maximizer on \(\mathcal K_1^2\).
\end{enumerate}
\end{lemma}

\begin{proof}
\begin{enumerate}
    \item Let \(K\in\mathcal K_1^m\) . For \(t\ge 0\), we set $K_t:=K+tB_1$ and
$\widetilde K_t:=|K_t|^{-1/m}K_t \in\mathcal K_1^m$. Since \(h_{K_t}(u)=h_K(u)+t\), we have 
\[
D(K_t)
=
\max_{u\in\mathbb S^{m-1}}
\bigl(h_{K_t}(u)+h_{K_t}(-u)\bigr)
=
\max_{u\in\mathbb S^{m-1}}
\bigl(h_K(u)+h_K(-u)+2t\bigr)
=
D(K)+2t.
\]

On the other hand, by Steiner's formula \eqref{eq:steiner}
\[
|K_t|
=
|K+tB_1|
=
\sum_{k=0}^m \binom{m}{k}W_k(K)t^k
=
1+mW_1(K)t+o(t).
\]
Hence
\[
D(\widetilde K_t)
=
|K_t|^{-1/m}D(K_t)
=
\bigl(1-W_1(K)t+o(t)\bigr)\bigl(D(K)+2t\bigr),
\]
so that
\[
D(\widetilde K_t)
=
D(K)+\bigl(2-W_1(K)D(K)\bigr)t+o(t).
\]

By the isoperimetric inequality, using \(mW_1(K)=P(K)\), and isodiametric inequality, we can write
$W_1(K)D(K)\ge 2$, where the equality holds if and only if \(K\) is a Euclidean ball. If \(K\) is not a ball, the inequality is strict; hence the above expansion gives \(D(\widetilde K_t)<D(K)\) for all sufficiently small \(t>0\). Hence the only local minimizer of the diameter on \(\mathcal K_1^m\) is the  Euclidean ball.

\item Let \(K\in\mathcal K_1^m\) be a local minimizer of the Cheeger constant \(h\) on \(\mathcal K_1^m\). We claim that \(K\) must be a Euclidean ball.

Assume by contradiction that \(K\) is not a ball. By Lemma~\ref{lem:css-cheeger-decay}, there exists a direction \(u\in\mathbb S^{m-1}\) such that, setting \(K_t:=S_u^t(K)\), one has \(h(K_t)<h(K)\) for every \(t\in(0,1]\). Since continuous Steiner symmetrization preserves convexity and volume, one has \(K_t\in\mathcal K_1^m\) for every \(t\in[0,1]\). Moreover, \(K_t\underset{t\rightarrow 0^+}{\longrightarrow} K\) in the Hausdorff metric; see \cite{Brock95}. Hence, for arbitrarily small \(t>0\), the body \(K_t\) lies arbitrarily close to \(K\) in \(\mathcal K_1^m\), while satisfying \(h(K_t)<h(K)\). This contradicts the local minimality of \(K\).
Therefore, balls are the only local minimizer of \(h\) on \(\mathcal K_1^m\).

\item Let \(K\in\mathcal K_1^m\). We show that \(K\) cannot be a local maximizer of the diameter. 

Choose \(x^+,x^-\in K\) such that
$
|x^+-x^-|=D(K)$ and let \(z\in\partial K\setminus\{x^+,x^-\}\), and let \(H\) be a supporting hyperplane to \(K\) at \(z\). For \(\delta>0\) sufficiently small, let \(H_\delta^{-}\) be the closed half-space obtained by translating \(H\) slightly towards the interior of \(K\), and set $K_\delta:=K\cap H_\delta^{-}$.

By construction, \(K_\delta\subset K\), \(K_\delta\underset{\delta\rightarrow0^+}{\longrightarrow} K\) in the Hausdorff metric, and \(|K_\delta|<1\). Moreover, for \(\delta\) small enough, the chosen diametral pair \(x^\pm\) still belongs to \(K_\delta\), and therefore \(D(K_\delta)=D(K)\). Defining
$\widetilde K_\delta:=|K_\delta|^{-1/m}K_\delta$,
one has \(\widetilde K_\delta\in\mathcal K_1^m\), \(\widetilde K_\delta\underset{\delta\rightarrow 0^+}{\longrightarrow} K\) in the Hausdorff metric, and
\[
\forall\,\delta>0,\qquad
D(\widetilde K_\delta)
=
|K_\delta|^{-1/m}D(K)
>
D(K).
\]
since \(|K_\delta|<1\). This excludes the existence of a local maximizer of \(D\) on \(\mathcal K_1^m\).

\item The fourth assertion (known only for the planar case) is proved in \cite[Proposition 4.10]{Ftouhi2025}. 
%The argument rests on \cite[Theorem~1.2]{Ftouhi2025} where non-tangential bodies are ruled out by the strict monotonicity of the scaling-invariant Cheeger constant of the parallel bodies of \(K\), whereas tangential bodies are excluded by combining the explicit formula for their Cheeger constant with the fact that the perimeter admits no local maximizer under area and convexity constraints. \ilias{ce n'est pas que ça, on utilise aussi mon inégalité. à mon avis il faut juste citer mon résultat et c'est bon. Pas la peine de résumer ce que j'ai fait}. 
\end{enumerate}

\end{proof}

\begin{remark}
The restriction to dimension \(m=2\) in the fourth assertion of Lemma \ref{lem:local-extrema}  comes from the
use of a result which is specific to planar convex sets. Indeed, in the plane, a characterization of Cheeger sets of convex domains due to Kawohl and
Lachand-Robert \cite{KawohlLachandRobert2006} is known and it is one of the main ingredients in the
perturbation argument of \cite[Proposition 4.10]{Ftouhi2025}. As far as we know, no generalization of this
argument is available for arbitrary dimensions
\(m\ge3\). 
\end{remark}

%========================================================
\section{Proof of Theorem \ref{thm:main-convex-bodies-diagram}: the diagram of convex bodies}\label{sec:thm_conv}
%========================================================

%The proof of Theorem~\ref{thm:main-convex-bodies-diagram} is distributed among Sections~3--6. The geometric estimates established in Lemma~\ref{lem:two-sided-cheeger-diameter} yield the growth bounds for the boundary functions and prove assertion~\textnormal{(v)}. The same estimates also provide the compactness needed for the attainment result in Proposition~\ref{prop:attainment-extrema}.

%The local extremal properties proved in Lemma~\ref{lem:local-extrema} are the main ingredients in the study of the boundary functions. Together with the attainment and continuity results, they yield Theorem~\ref{thm:boundary-functions}, which proves assertions~\textnormal{(ii)}--\textnormal{(iv)}. Assertion~\textnormal{(vi)} is proved in Proposition~\ref{prop:first-order-lower-boundary}.

%Finally, Proposition~\ref{prop:closedness-diagram} establishes the closedness of the diagram, while Theorem~\ref{thm:simple-connectedness-convex-bodies} proves the representation formula in assertion~\textnormal{(i)}. The simple connectedness of \(\mathcal D_m\) follows from this representation.\ilias{Ce paragraphe va changer quand tu changeras l'ordre des preuves}

%========================================================
\subsection{Study of the boundary of the diagram}
%========================================================

% \ilias{Ici il faut commencer par rappeler la définition du diagramme pour les convex bodies.}

% In this section, we prove that the extrema defining """the boundary functions""" \ilias{what are the boundary functions ? Please avoid using such non-precise references.}
% are attained, establish the continuity of both boundary functions and their
% strict monotonicity in the cases stated in
% Theorem~\ref{thm:main-convex-bodies-diagram}, determine the first-order
% behavior of the lower boundary at the ball, and derive equivalent
% formulations of the corresponding optimization problems.

Recall that the diagram under study is the following
\[
\mathcal D_m
:=
\left\{
\bigl(D(K),h(K)\bigr):
K\in\mathcal K_1^m
\right\}.
\]

In this subsection, we first prove that the extrema defining the boundary
functions \(f\) and \(g\) of the diagram, defined by
\eqref{eq:def-boundary-functions}, are attained. We then establish the
continuity of these functions and the strict monotonicity statements appearing
in Theorem~\ref{thm:main-convex-bodies-diagram}. After that, we derive
equivalent formulations of the corresponding optimization problems. Finally,
we determine the first-order behavior of the lower boundary near the point corresponding to the ball.
\begin{proposition}\label{prop:attainment-extrema}
Let \(B\subset\mathbb R^m\) be a ball of volume \(1\). The following assertions hold.
\begin{itemize}
    \item For every \(d\ge D(B)\), the extrema $\inf\backslash\sup\{h(K):K\in\mathcal K_1^m,\ D(K)=d\}$ are attained.

\item  For every \(h_*>h(B)\), the extrema $\inf\backslash\sup\{D(K):K\in\mathcal K_1^m,\ h(K)=h_*\}$
are attained.
\end{itemize}

\end{proposition}

\begin{proof}
The result follows from Blaschke's selection theorem together with the continuity of the functionals \(D\) and \(h\).

In the case of fixed diameter, compactness follows directly from the diameter constraint. In the case of fixed Cheeger constant, compactness is a consequence of Lemma~\ref{lem:two-sided-cheeger-diameter}: if \(h(K)=h_*\) and \(|K|=1\), then $D(K)\le m/\omega_{m-1}h_*^{m-1}$.
Therefore, every minimizing or maximizing sequence is uniformly bounded up to translations. The conclusion follows by the direct method.
\end{proof}

\begin{proposition}\label{prop:growth-boundary-functions}
Let \(B\) be the Euclidean ball of volume \(1\), and set \(d_0:=D(B)\).
Let \(f\) and \(g\) be defined by \eqref{eq:def-boundary-functions}. Then there
exist constants \(c_m,C_m>0\), depending only on \(m\), such that, for every
\(x>d_0\), one has
\(c_m x^{1/(m-1)}\le f(x)\le g(x)\le C_m x^{m-1}\).
Moreover, the exponents \(1/(m-1)\) and \(m-1\) are sharp.
\end{proposition}

\begin{proof}
Let \(x>d_0\), and let \(K\in\mathcal K_1^m\) be such that \(D(K)=x\).
Since \(|K|=1\), Lemma~\ref{lem:two-sided-cheeger-diameter} gives
\[
\left(\frac{\omega_{m-1}}{m}\right)^{1/(m-1)}
x^{1/(m-1)}
\le h(K)\le
\frac{m\,\omega_m}{2^{m-1}}x^{m-1}.
\]
Taking the infimum over all \(K\in\mathcal K_1^m\) such that \(D(K)=x\), we get
\(c_m x^{1/(m-1)}\le f(x)\). Taking the supremum over the same class gives
\(g(x)\le C_m x^{m-1}\). Inequality \(f(x)\le g(x)\) follows from the
definitions of \(f\) and \(g\). The sharpness of the exponents follows from the optimality statement in Lemma~\ref{lem:two-sided-cheeger-diameter}.
\end{proof}
\begin{proposition} \label{prop:boundary-functions}
Let \(\mathcal K_1^m\) denote the class of convex bodies in \(\mathbb R^m\) with volume \(1\), and let \(\mathcal D_m\) be the associated Blaschke--Santal\'o diagram. Let \(B\) be the Euclidean ball of volume \(1\), and set \(d_0:=D(B)\). For every \(x\ge d_0\), define
\[
f(x):=\inf\bigl\{h(K):K\in\mathcal K_1^m,\ D(K)=x\bigr\},
\qquad
g(x):=\sup\bigl\{h(K):K\in\mathcal K_1^m,\ D(K)=x\bigr\}.
\]
Then, the following assertions hold.
\begin{enumerate}[label=\textnormal{(\roman*)}]
\item The functions \(f\) and \(g\) are continuous on \([d_0,+\infty)\).
\item The function \(f\) is strictly increasing on \([d_0,+\infty)\).
\item If \(m=2\), then \(g\) is strictly increasing on \([d_0,+\infty)\).
\end{enumerate}
\end{proposition}

\begin{proof}
Let us present the proof of each assertion. 
\begin{enumerate}[label=\textnormal{(\roman*)}]
\item Let \((x_n)\subset [d_0,+\infty)\) be such that \(x_n\to x_0\). 

$\bullet$ We begin by proving the lower limit inequality
\[
f(x_0)\le \liminf_{n\to\infty} f(x_n).
\]

Choose a subsequence \((x_{n_k})\) such that
\(f(x_{n_k}) \to \liminf_{n\to\infty} f(x_n)\).
For each \(k\), let \(K_k\in\mathcal K_1^m\) be a minimizer satisfying
\(D(K_k)=x_{n_k}\) and \(h(K_k)=f(x_{n_k})\).
Since \(x_{n_k}\to x_0\), the diameters \(D(K_k)\) are uniformly bounded. As \(|K_k|=1\), Blaschke's selection theorem applies up to translations. Hence, we may assume that \(K_k\underset{k\rightarrow+\infty}{\longrightarrow} K_*\) in the Hausdorff metric for some convex body \(K_*\).

By the continuity of the volume, the diameter, and the Cheeger constant with respect to the Hausdorff distance in the class of convex bodies ; see
Subsection~\ref{subsec:convex-bodies-notation} and
\cite[Proposition~3.1]{Parini2017},
\(|K_*|=1\), \(D(K_*)=x_0\), and \(h(K_*)=\lim_{k\to\infty} h(K_k)\).
Thus \(K_*\in\mathcal K_1^m\) and \(D(K_*)=x_0\). By the definition of \(f(x_0)\), we obtain
\(f(x_0)\le h(K_*)\).
Since \(h(K_k)=f(x_{n_k})\), it follows that
\(f(x_0)\le h(K_*)=\lim_{k\to\infty} f(x_{n_k})=\liminf_{n\to\infty} f(x_n)\).

$\bullet$ Let us now show the upper limit inequality
$$\limsup_{n\to\infty} f(x_n)\le f(x_0).$$

Let \(K_0\in\mathcal K_1^m\) be a minimizer for \(f(x_0)\), so that \(D(K_0)=x_0\) and \(h(K_0)=f(x_0)\). By the diameter perturbations used in the proof of Lemma~\ref{lem:local-extrema}, there exists a sequence \(L_n\in\mathcal K_1^m\) such that \(L_n\to K_0\) in the Hausdorff metric and \(D(L_n)=x_n\) for every \(n\). Since \(L_n\) is admissible for \(f(x_n)\), we have \(f(x_n)\le h(L_n)\). Passing to the upper limit and using the continuity of \(h\), we obtain
\[
\limsup_{n\to\infty} f(x_n)
\le \lim_{n\to\infty} h(L_n)
= h(K_0)
= f(x_0).
\]

Together with the lower limit inequality, this gives \(f(x_n)\to f(x_0)\). Hence \(f\) is continuous at \(x_0\). and since \(x_0\) was arbitrary, \(f\) is continuous on \([d_0,+\infty)\).
The continuity of \(g\) follows by the same argument, replacing minimizers by maximizers throughout.

\item Let us prove that \(f\) is strictly increasing. Assume by contradiction that this is not the case. Since \(f\) is continuous on \([d_0,+\infty)\), and since \(f(x)>f(d_0)=h(B)\) for every \(x>d_0\), the function \(f\) must attain a local minimum at some point \(x_0>d_0\).

Let \(K_0\in\mathcal K_1^m\) be such that \(D(K_0)=x_0\) and \(h(K_0)=f(x_0)\). By local minimality of \(x_0\), there exists \(\varepsilon>0\) such that \(f(x_0)\le f(x)\) for every \(x\in(x_0-\varepsilon,x_0+\varepsilon)\). Hence, if \(K\in\mathcal K_1^m\) and \(D(K)\in(x_0-\varepsilon,x_0+\varepsilon)\), then
\[
h(K_0)=f(x_0)\le f(D(K))\le h(K).
\]

Since the diameter is continuous with respect to Hausdorff convergence, this shows that \(K_0\) is a local minimizer of the Cheeger constant on \(\mathcal K_1^m\). By the second assertion of Lemma~\ref{lem:local-extrema}, \(K_0\) must be the Euclidean ball. This is impossible, since \(D(K_0)=x_0>d_0=D(B)\). The contradiction proves that \(f\) is strictly increasing.
\item Assume now that \(m=2\) and that \(g\) is not strictly increasing. Since \(g\) is continuous on \([d_0,+\infty)\), and since \(g(x)>g(d_0)=h(B)\) for every \(x>d_0\), the function \(g\) must attain a local maximum at some point \(x_0>d_0\).

Choose \(K_0\in\mathcal K_1^2\) such that \(D(K_0)=x_0\) and \(h(K_0)=g(x_0)\). By local maximality of \(x_0\), there exists \(\varepsilon>0\) such that \(g(x)\le g(x_0)\) for every \(x\in(x_0-\varepsilon,x_0+\varepsilon)\). Therefore, if \(K\in\mathcal K_1^2\) and \(D(K)\in(x_0-\varepsilon,x_0+\varepsilon)\), then
\[
h(K)\le g(D(K))\le g(x_0)=h(K_0).
\]
It follows, again by the continuity of the diameter, that \(K_0\) is a local maximizer of the Cheeger constant on \(\mathcal K_1^2\). This contradicts the fourth assertion of Lemma~\ref{lem:local-extrema}. Hence \(g\) is strictly increasing in the planar case.
\end{enumerate}

\end{proof}
\begin{remark}
The proof of assertion~\textnormal{(iii)} shows that, for a fixed
\(m\ge2\), the strict monotonicity of the function \(g\) would follow from the eventual nonexistence of
local maximizers of the Cheeger constant on \(\mathcal K_1^m\), that, as far as we know, remains open.  
\end{remark}
\begin{corollary}[Equivalent formulations of the boundary problems] 
Let \(x>d_0\). In every dimension \(m\ge 2\), the following problems are equivalent:

\medskip

\noindent
\begin{tabular}{@{}p{0.48\textwidth}p{0.48\textwidth}@{}}
\textnormal{(i)} \(\min\bigl\{h(K):K\in\mathcal K_1^m,\ D(K)=x\bigr\}\);
&
\textnormal{(ii)} \(\min\bigl\{h(K):K\in\mathcal K_1^m,\ D(K)\ge x\bigr\}\);
\\[1.2em]
\textnormal{(iii)} \(\max\bigl\{D(K):K\in\mathcal K_1^m,\ h(K)=f(x)\bigr\}\);
&
\textnormal{(iv)} \(\max\bigl\{D(K):K\in\mathcal K_1^m,\ h(K)\le f(x)\bigr\}\).
\end{tabular}

\medskip

If \(m=2\), the following problems are also equivalent:

\medskip

\noindent
\begin{tabular}{@{}p{0.48\textwidth}p{0.48\textwidth}@{}}
\textnormal{(i)} \(\max\bigl\{h(K):K\in\mathcal K_1^2,\ D(K)=x\bigr\}\);
&
\textnormal{(ii)} \(\max\bigl\{h(K):K\in\mathcal K_1^2,\ D(K)\le x\bigr\}\);
\\[1.2em]
\textnormal{(iii)} \(\min\bigl\{D(K):K\in\mathcal K_1^2,\ h(K)=g(x)\bigr\}\);
&
\textnormal{(iv)} \(\min\bigl\{D(K):K\in\mathcal K_1^2,\ h(K)\ge g(x)\bigr\}\).
\end{tabular}
\end{corollary}
\begin{proof}
The proof is the same monotonicity argument as in the proof of
\cite[Corollary~3.13]{FtouhiLamboley2021}. For the first four problems,
one uses the definition of \(f\) together with its monotonicity.
The equivalence between the second four problems follows in the same way,
using the definition of the function \(g\) and its monotonicity in dimension two.
\end{proof}

\begin{proposition}[First-order behavior of the lower boundary near the ball]\label{prop:first-order-lower-boundary}
Let \(B\) be the Euclidean ball of volume \(1\), and set \(d_0:=D(B)\). Then the lower boundary function \(f\) has a right derivative at \(d_0\), and
$
f'_+(d_0)=0$.
\end{proposition}
\begin{proof}
Let \(B=B(0,R)\) be the Euclidean ball of volume \(1\), so that \(d_0=D(B)=2R\). For \(t\ge 0\), consider the volume-preserving linear map
\[
A_t:=\operatorname{diag}\left(1+t,(1+t)^{-1/(m-1)},\ldots,(1+t)^{-1/(m-1)}\right),
\]
and set \(K_t=A_tB\). Since \(\det A_t=1\), we have \(|K_t|=1\). Moreover, the diameter of \(K_t\) is attained in the \(e_1\)-direction, and therefore
$D(K_t)=(1+t)d_0$. Moreover, if we set \(x_t:=D(K_t)\), we have \(x_t-d_0=d_0t\).

Let us now estimate \(P(K_t)\). By the area formula for linear images of hypersurfaces, see \cite[Lemma~1.4]{Schmidt2015}, applied to \(x\mapsto A_tx\), we can write
\[
P(K_t)=\int_{\partial B}\left|\operatorname{cof}(A_t)\nu_B\right|\,d\mathcal H^{m-1},
\]
where \(\nu_B\) is the outer unit normal to \(\partial B\). Since \(\det A_t=1\), \(\operatorname{cof}(A_t)=A_t^{-T}\).

Since \(\partial B=R\cdot\mathbb S^{m-1}\) and \(\nu_B(R\omega)=\omega\), we obtain
\[
P(K_t)
=
R^{m-1}
\int_{\mathbb S^{m-1}}
|A_t^{-T}\omega|\,d\mathcal H^{m-1}.
\]

Now
\[
A_t^{-T}
=
\operatorname{diag}\left((1+t)^{-1},(1+t)^{1/(m-1)},\ldots,(1+t)^{1/(m-1)}\right).
\]

Thus, for \(\omega=(\omega_1,\omega_2,\ldots,\omega_m)\in\mathbb S^{m-1}\), we have
\[
|A_t^{-T}\omega|
=
1+t\left(
\frac{1}{m-1}
-
\frac{m}{m-1}\omega_1^2
\right)
+o(t),
\]
uniformly on \(\mathbb S^{m-1}\). Therefore
\[
P(K_t)
=
P(B)
+
tR^{m-1}
\int_{\mathbb S^{m-1}}
\left(
\frac{1}{m-1}
-
\frac{m}{m-1}\omega_1^2
\right)
d\mathcal H^{m-1}
+
o(t).
\]

The integral in the coefficient of \(t\) vanishes. Indeed, by rotational symmetry,
\[
\int_{\mathbb S^{m-1}}\omega_1^2\,d\mathcal H^{m-1}
=
\frac1m\mathcal H^{m-1}(\mathbb S^{m-1}).
\]

Hence $$P(K_t)=P(B)+o(t).$$

Since \(K_t\in\mathcal K_1^m\) and \(D(K_t)=x_t\), it is admissible for \(f(x_t)\), we have
$ f(x_t)\le h(K_t)$.

Also \(f(d_0)=h(B)\), because the ball is the unique body of volume \(1\) and diameter \(d_0\). Moreover,
$h(K_t)\le P(K_t)$, and 
$h(B)=P(B)$, because \(|K_t|=|B|=1\). Consequently,
\[
0\le f(x_t)-f(d_0)
\le h(K_t)-h(B)
\le P(K_t)-P(B)
=o(t).
\]

Finally, \(x_t-d_0=d_0t\). Hence
\[
0
\le
\frac{f(x_t)-f(d_0)}{x_t-d_0}
\le
o(1).
\]
Since $x_t\underset{t\to 0^+}{\longrightarrow} d_0$, we obtain
$f'_+(d_0)=0$, which completes the proof.
\end{proof}

%========================================================
\subsection{Simple connectedness of the diagram}
%========================================================
In this subsection, we prove that \(\mathcal D_m\) is closed and that it is
exactly the region between the graphs of the functions \(f\) and \(g\) defined
in \eqref{eq:def-boundary-functions}. The simple
connectedness of the diagram follows from this description.
%\subsection{Closedness of the diagram}

Let us first prove that \(\mathcal D_m\) is closed by using Blaschke's selection
theorem and the continuity of volume, diameter, and the Cheeger constant
under Hausdorff convergence.

\begin{proposition}\label{prop:closedness-diagram}
Let $m\ge 2$. The Blaschke--Santal\'o diagram
$$
\mathcal{D}_m:=\{(D(K),h(K)):\ K\in\mathcal{K}_1^m\}
$$
is a closed subset of $\mathbb{R}^2$.
\end{proposition}

%\ilias{Je dirais que Proposition \ref{prop:attainment-extrema} est un corollaire de cette proposition. N'est ce pas le cas ? si oui, il faut donc supprimer une preuve pour eviter la redendance}
%{\color{red}
%Zakaria dit : Oui, tu as raison. Une fois la Proposition~\ref{prop:closedness-diagram} démontrée, Proposition~\ref{prop:attainment-extrema} peut être obtenue comme conséquence des arguments de compacité. Pour le diamètre fixé, la compacité suit directement du théorème de sélection de Blaschke. Pour \(h(K)=h_*\), il faut d'abord borner le diamètre; cela découle du Lemme~\ref{lem:two-sided-cheeger-diameter}, puisque
%\[
%h_*\ge c_m D(K)^{1/(m-1)},
%\qquad\text{donc}\qquad
%D(K)\le \left(\frac{h_*}{c_m}\right)^{m-1}.
%\]
%On peut alors appliquer le même argument de compacité. Je préfère néanmoins
%garder la proposition sous sa forme actuelle, afin de regrouper les propriétés topologiques du diagramme dans une sous-section séparée.
%}

\begin{proof}
\setcounter{proofstep}{0}
We begin with closedness. Let $(d_n,c_n)$ be a sequence in $\mathcal{D}_m$ converging in $\mathbb{R}^2$ to some point $(d,c)$. By the definition of the diagram, for each $n\ge 1$ there exists a body $K_n\in\mathcal{K}_1^m$ such that
$D(K_n)=d_n$ and $h(K_n)=c_n$.

Since the sequence $(D(K_n))$ is bounded and $|K_n|=1$ for every $n$, the family $(K_n)$ is uniformly bounded up to translations. We then may assume that all bodies are contained in a common compact ball of $\mathbb{R}^m$. By Blaschke's selection theorem, there exists a subsequence, still denoted by $(K_n)$, and a convex body $K\in\mathcal{K}^m$ such that $K_n\underset{n\rightarrow+\infty}{\longrightarrow} K$ in the Hausdorff metric; see, For example, ~\cite[Theorem.~1.8.7]{Schneider2014}.

Since the volume is continuous on $\mathcal{K}^m$ with respect to Hausdorff convergence, we obtain $|K|=1$, hence $K\in\mathcal{K}_1^m$. Moreover, the diameter is continuous on $\mathcal{K}^m$, so $D(K)=\lim_{n\to\infty}D(K_n)=d$.

At this point we use the continuity of the Cheeger constant on $\mathcal{K}_1^m$ with respect to Hausdorff convergence see~\cite[Prop.~3.1]{Parini2017}. Therefore
$h(K)=\lim_{n\to\infty}h(K_n)=c$, which yields  $(d,c)=(D(K),h(K))\in\mathcal{D}_m$, 
which proves that $\mathcal{D}_m$ is closed.
\end{proof}
%\subsection{Simple connectedness of the diagram}

We are now ready to prove the simple connectedness of the diagram. To do so, we show that every point between the graphs of \(f\) and \(g\) belongs to
\(\mathcal D_m\). The proof uses normalized Minkowski interpolation and a
winding-number argument.
\begin{theorem}
\label{thm:simple-connectedness-convex-bodies}
Let \(B\subset\mathbb R^m\) be the Euclidean ball of volume \(1\). For every \(m\ge2\), the Blaschke--Santaló diagram
\[
\mathcal D_m
=
\bigl\{(x,y)\in\mathbb R^2:\ x\ge D(B),\ f(x)\le y\le g(x)\bigr\}
\]
is simply connected.
\end{theorem}

\begin{proof}
%__%\proofstep{Construction judicious  closed curves.}{step:closed-curves}
\textbf{Step 1: Construction judicious  closed curves.}
Let $K,L\in \mathcal{K}_1^m$ satisfying $D(K)=D(L)$. For $t\in[0,1]$, we consider the Minkowski interpolation
$M_t:=(1-t)K+tL$
and normalize it with respect to its volume as follows
$
\widetilde M_t:=|M_t|^{-1/m}M_t \in \mathcal{K}_1^m$.
% Since Minkowski addition preserves convexity and $|\widetilde M_t|=1$ by construction, one has
% \[
% \widetilde M_t\in \mathcal{K}_1^m
% \qquad\text{for every }t\in[0,1].
% \]

We then define a path in the Blaschke--Santal\'o diagram  \(\mathcal{D}_m \) by
$\bigl\{(D(\widetilde M_t),\,h(\widetilde M_t)):\ t\in[0,1]\bigr\}$ that
we close with an auxiliary vertical segment in the common abscissa \(D(K)=D(L)\). This segment is introduced solely to obtain a closed planar curve; it is not assumed to be contained in the Blaschke--Santal\'o diagram. The resulting curve will be used below through its winding number with respect to a relevant point that will be introduced later. More precisely, define
\[
\sigma_{K,L}:t\in[0,2]\longmapsto
\begin{cases}
\bigl(D(\widetilde M_t),\,h(\widetilde M_t)\bigr), & t\in[0,1],\\[4pt]
\bigl(D(K),\,(2-t)\,h(L)+(t-1)\,h(K)\bigr), & t\in[1,2].
\end{cases}
\]
%Thus \(\sigma_{K,L}\) is a closed continuous curve in \(\mathbb R^2\).

% Indeed,
% \[
% \sigma_{K,L}(0)=(D(K),h(K)),
% \qquad
% \sigma_{K,L}(1)=(D(L),h(L)),
% \qquad
% \sigma_{K,L}(2)=(D(K),h(K)),
% \]
% and the assumption $D(K)=D(L)$ ensures that the two branches match continuously at $t=1$.

% It remains to justify continuity. The map $t\mapsto M_t$ is continuous in the Hausdorff metric, since for every $u\in \mathbb S^{m-1}$ one has
% \[
% h_{M_t}(u)=(1-t)h_K(u)+t\,h_L(u).
% \]
% Together with the continuity of $t\mapsto |M_t|$, this implies the continuity of $t\mapsto \widetilde M_t$ in $\mathcal K_1^m$. Since both the diameter and the Cheeger constant are continuous on $\mathcal K_1^m$, the first branch of $\sigma_{K,L}$ is continuous on $[0,1]$, while the second branch is manifestly continuous on $[1,2]$. Therefore $\sigma_{K,L}$ is a continuous closed curve in $\mathcal D_m$.

%__%\proofstep{Uniform continuity of the closed curves.}{step:uniform-convergence-curves}
\textbf{Step 2. Uniform continuity of the closed curves.}
Let $K,L\in\mathcal K_1^m$ be such that $D(K)=D(L)=d_0$, and let $(K_n)$ and $(L_n)$ be sequences in $\mathcal K_1^m$ converging to $K$ and $L$, respectively, in the Hausdorff metric, with
$D(K_n)=D(L_n)\ \text{for every }n\ge 1.$

We show that the associated closed curves $\sigma_{K_n,L_n}$ converge uniformly to $\sigma_{K,L}$ on $[0,2]$ as $n\to +\infty$. More precisely, for every $\varepsilon>0$, there exists $N_\varepsilon\in\mathbb N$ such that
\[
\sup_{t\in[0,2]}
\bigl\|\sigma_{K_n,L_n}(t)-\sigma_{K,L}(t)\bigr\|
\le \varepsilon
\qquad\text{for all }n\ge N_\varepsilon.
\]
 
  For $t\in[0,1]$, set $M_t^n:=(1-t)K_n+tL_n$ and $M_t:=(1-t)K+tL$
and consider their volume-normalized versions $\widetilde M_t^n:=|M_t^n|^{-1/m}M_t^n$ and 
$\widetilde M_t:=|M_t|^{-1/m}M_t$. We then have
\[
\|\sigma_{K_n,L_n}(t)-\sigma_{K,L}(t)\|
\le
|D(\widetilde M_t^n)-D(\widetilde M_t)|
+
|h(\widetilde M_t^n)-h(\widetilde M_t)|.
\]

We estimate the two terms separately. Since the diameter is $2$-Lipschitz with respect to the Hausdorff distance, one has
\[
|D(\widetilde M_t^n)-D(\widetilde M_t)|
\le
2\,d_H(\widetilde M_t^n,\widetilde M_t).
\]

We begin by estimating the Hausdorff distance between $\widetilde M_t^n$ and $\widetilde M_t$. Since
$h_{\widetilde M_t^n}=|M_t^n|^{-1/m}h_{M_t^n}$ and $
h_{\widetilde M_t}=|M_t|^{-1/m}h_{M_t}$, 
we may write, for every $u\in\mathbb S^{m-1}$,
\[
h_{\widetilde M_t^n}(u)-h_{\widetilde M_t}(u)
=
|M_t^n|^{-1/m}\bigl(h_{M_t^n}(u)-h_{M_t}(u)\bigr)
+
\bigl(|M_t^n|^{-1/m}-|M_t|^{-1/m}\bigr)\,h_{M_t}(u).
\]

Taking the supremum over $u\in\mathbb S^{m-1}$, we obtain
\[
d_H(\widetilde M_t^n,\widetilde M_t)
\le
|M_t^n|^{-1/m}\,\|h_{M_t^n}-h_{M_t}\|_\infty
+
\|h_{M_t}\|_\infty\,\bigl||M_t^n|^{-1/m}-|M_t|^{-1/m}\bigr|.
\]

The first term is controlled by using the linearity of the support function with respect to Minkowski sums, i.e.,
$h_{M_t^n}=(1-t)h_{K_n}+t\,h_{L_n}$, and 
$h_{M_t}=(1-t)h_K+t\,h_L$. Hence,
\[
\|h_{M_t^n}-h_{M_t}\|_\infty
\le
(1-t)\|h_{K_n}-h_K\|_\infty+t\|h_{L_n}-h_L\|_\infty
\le
\|h_{K_n}-h_K\|_\infty+\|h_{L_n}-h_L\|_\infty.
\]

We next estimate the second term. By the Brunn--Minkowski inequality, we have
\[
|M_t^n|^{1/m}\ge (1-t)|K_n|^{1/m}+t|L_n|^{1/m}=1,
\ \text{and},\ 
|M_t|^{1/m}\ge (1-t)|K|^{1/m}+t|L|^{1/m}=1.
\]
% In particular,
%\[
%|M_t^n|^{-1/m}\le 1.
%\]

Moreover, by applying the mean value theorem to the function $x\mapsto x^{-1/m}$, we obtain
\[
\bigl||M_t^n|^{-1/m}-|M_t|^{-1/m}\bigr|
\le
{\scriptstyle \sup\limits_{x\ge 1}\left|-\frac1m x^{-1/m-1}\right|}
\cdot\bigl||M_t^n|-|M_t|\bigr|
\le
\frac1m\,\bigl||M_t^n|-|M_t|\bigr|,
\]
where we have used the fact that\(|M_t^n|, |M_t|\ge 1\) for every \(t\in[0,1]\).

It remains to estimate the difference between the volumes. Expanding $|M_t^n|$ and $|M_t|$ by Minkowski's polynomial formula \eqref{eq:minko}, we obtain 
\[
\bigl||M_t^n|-|M_t|\bigr|
\le
\sum_{k=0}^m \binom{m}{k}(1-t)^{m-k}t^k\,
\bigl|W_k(K_n,L_n)-W_k(K,L)\bigr|
\le \sum_{k=0}^m \binom{m}{k}\,
\bigl|W_k(K_n,L_n)-W_k(K,L)\bigr|.
\]

Using 
\[
\|h_{M_t}\|_\infty
=
\|(1-t)h_K+t h_L\|_\infty
\le
(1-t)\|h_K\|_\infty+t\|h_L\|_\infty
\le
\|h_K\|_\infty+\|h_L\|_\infty,
\]
we obtain
\begin{equation}\label{eq:uniform-hausdorff-bound}
\begin{aligned}
d_H(\widetilde M_t^n,\widetilde M_t)
&\le
d_H(K_n, K)+d_H( L_n, L) \\
&\quad
+\frac{\|h_K\|_\infty+\|h_L\|_\infty}{m}
\sum_{k=0}^m \binom{m}{k}
\bigl|W_k(K_n,L_n)-W_k(K,L)\bigr|.
\end{aligned}
\end{equation}

Since $K_n\underset{n\rightarrow+\infty}{\longrightarrow} K$ and $L_n\underset{n\rightarrow+\infty}{\longrightarrow} L$ in the Hausdorff metric, the mixed volumes $W_k(K_n,L_n)$ converge to $W_k(K,L)$; for every $k\in \llbracket 0,m \rrbracket$ ; see, e.g.,~\cite[Section~5.1]{Schneider2014}. It follows that
\begin{equation}\label{eq:uniform-hausdorff-convergence}
\sup_{t\in[0,1]} d_H(\widetilde M_t^n,\widetilde M_t)\underset{n\rightarrow+\infty}{\longrightarrow} 0.
\end{equation}

It remains to control the term
$\bigl|h(\widetilde M_t^n)-h(\widetilde M_t)\bigr|$. Set
$
d_0:=D(K)=D(L)
$.
Since $K_n\to K$ and $L_n\to L$ in the Hausdorff metric, we may assume, for $n$ sufficiently large, that
$D(K_n),\,D(L_n)\le d_0+1$.

For every $t\in[0,1]$, the subadditivity of the diameter under Minkowski addition gives
\[
D(M_t^n)\le (1-t)D(K_n)+tD(L_n)\le d_0+1,
\]
and similarly
\[
D(M_t)\le (1-t)D(K)+tD(L)=d_0.
\]

Since $|M_t^n|^{1/m},\,|M_t|^{1/m}\ge 1$, by the Brunn--Minkowski inequality, it follows that
\[
D(\widetilde M_t^n)\le D(M_t^n)\le d_0+1,
\qquad
D(\widetilde M_t)\le D(M_t)\le d_0+1
\]
for every $t\in[0,1]$ and all $n$ sufficiently large.

Since \(|\widetilde M_t^n|=|\widetilde M_t|=1\) and
\(
D(\widetilde M_t^n),\,D(\widetilde M_t)\le d_0+1,
\)
the inradius estimate obtained in Lemma  \ref{lem:lower_r_d_A} gives
$$
r(\widetilde M_t^n),\,r(\widetilde M_t)
\ge
\frac{2^{m-1}}{m\,\omega_m\,(d_0+1)^{m-1}}=:\rho_0.
$$

Applying Lemma~\ref{lem:cox-monotone-homogeneous}, we obtain
\[
|h(\widetilde M_t^n)-h(\widetilde M_t)|
\le
\frac{m}{\rho_0^2}\,d_H(\widetilde M_t^n,\widetilde M_t).
\]

Combining this estimate with \eqref{eq:uniform-hausdorff-bound}, we infer that, for every $t \in [0,1]$
\[
|h(\widetilde M_t^n)-h(\widetilde M_t)|
\le
\frac{m}{\rho_0^2}\,\Bigg(
d_H(K_n, K)+d_H( L_n, L) \\
+\frac{\|h_K\|_\infty+\|h_L\|_\infty}{m}
\sum_{k=0}^m \binom{m}{k}
\bigl|W_k(K_n,L_n)-W_k(K,L)\bigr|\Bigg).
\]

Since the right-hand side is independent of $t$, we conclude that
\[
\sup_{t\in[0,1]} |h(\widetilde M_t^n)-h(\widetilde M_t)|\underset{n\rightarrow+\infty}{\longrightarrow} 0.
\]

Combining this with the uniform convergence of \eqref{eq:uniform-hausdorff-convergence}, we obtain
\[
\sup_{t\in[0,1]}
\bigl\|\sigma_{K_n,L_n}(t)-\sigma_{K,L}(t)\bigr\|
\underset{n\rightarrow+\infty}{\longrightarrow} 0.
\]

 For the range $t\in [1,2]$, the claim follows directly from the parametrization of the auxiliary vertical branch, i.e.,
\[
\sigma_{K_n,L_n}(t)
=
\bigl(D(K_n),\,(2-t)\,h(L_n)+(t-1)\,h(K_n)\bigr),
\]
and
\[
\sigma_{K,L}(t)
=
\bigl(D(K),\,(2-t)\,h(L)+(t-1)\,h(K)\bigr).
\]

We write
\[
\begin{aligned}
\bigl\|\sigma_{K_n,L_n}(t)-\sigma_{K,L}(t)\bigr\|
&\le
|D(K_n)-D(K)| 
+(2-t)\,|h(L_n)-h(L)|
+(t-1)\,|h(K_n)-h(K)|\\
&\leq |D(K_n)-D(K)|
+ |h(K_n)-h(K)| +|h(L_n)-h(L)|.
\end{aligned}
\]

The right-hand side is independent of $t$ and tends to zero as $n\to\infty$, by convergence $K_n\underset{n\rightarrow+\infty}{\longrightarrow} K$, $L_n\underset{n\rightarrow+\infty}{\longrightarrow} L$, and the continuity of the diameter and the Cheeger constant in $\mathcal K_1^m$. 

We then have shown the claimed uniform continuity, i.e.,
\[
\sup_{t\in[0,2]}
\bigl\|\sigma_{K_n,L_n}(t)-\sigma_{K,L}(t)\bigr\|
\underset{n\rightarrow+\infty}{\longrightarrow} 0.
\]

%__%\proofstep{Cylindrical classes and closedness.}{step:cylindrical-classes}

\textbf{Step 3. A special class of sets.}

For \(R>0\) and \(d>0\), define
\[
\mathcal C_R(d):=
\left\{
K\in\mathcal K_1^m:
D(K)=d,\ 0\in K,\ de_1\in K,\ 
K\subset [0,d]\times B_{m-1}\left(0,Rd^{-1/(m-1)}\right)
\right\}.
\]
Here \(e_1\) denotes the first vector of the canonical basis of \(\mathbb R^m\),
and \(B_{m-1}(0,r)\) denotes the Euclidean ball of radius \(r\) in
\(\mathbb R^{m-1}\). The conditions \(0\in K\) and \(de_1\in K\) fix a
diametral segment of \(K\). Moreover,
\[
\left|
[0,d]\times B_{m-1}\left(0,Rd^{-1/(m-1)}\right)
\right|
=
\omega_{m-1}R^{m-1}.
\]

We record the following closedness property. Let \(R>0\), let \(d_n\to d>0\),
and let $K_n\in\mathcal C_R(d_n)$. Assume that \(K_n\to K\) in the Hausdorff metric. Then $K\in\mathcal C_R(d)$.

Indeed, the limit \(K\) is in the class $K_1^m$. By the continuity of the volume
and the diameter with respect to the Hausdorff convergence, 
\[
|K|=\lim_{n\to\infty}|K_n|=1,
\qquad
D(K)=\lim_{n\to\infty}D(K_n)=\lim_{n\to\infty}d_n=d.
\]

 Moreover, since \(0\in K_n\) for every
\(n\), by the Hausdorff convergence, we also have $0\in K$. Similarly, since \(d_ne_1\in K_n\) and \(d_ne_1\to de_1\), we obtain
$de_1\in K$.

Let \(x\in K\). By the Hausdorff convergence, there exist \(x_n\in K_n\) such that
\(x_n\to x\). Write
$
x_n=(x_{n,1},x_n')$ and
$x=(x_1,x')$, with \(x_n',x'\in\mathbb R^{m-1}\). Since \(K_n\in\mathcal C_R(d_n)\), we have
$
0\le x_{n,1}\le d_n$ and $
|x_n'|\le R d_n^{-1/(m-1)}$. Passing to the limit, we show that 
$0\le x_1\le d$ and $
|x'|\le R d^{-1/(m-1)}$, which means that 
$x\in [0,d]\times B_{m-1}\left(0,Rd^{-1/(m-1)}\right)$.

Since the choice of \(x\in K\) was arbitrary, we get the inclusion
$
K\subset [0,d]\times B_{m-1}\left(0,Rd^{-1/(m-1)}\right)$, which means that \(K\in\mathcal C_R(d)\).

%__%\proofstep{Stability under volume-preserving affine images.}{step:volume-preserving-affine-images}

\textbf{Step 4. Behavior under volume-preserving affine maps.}

Let \(R>0\), \(d>0\), \(K\in\mathcal C_R(d)\), and \(\lambda\ge1\). Define
\(T_\lambda(x_1,x'):=(\lambda x_1,\lambda^{-1/(m-1)}x')\), with
\(x'\in\mathbb R^{m-1}\). Let us prove that \(T_\lambda K\in\mathcal C_R(\lambda d)\).

Since \(\det T_\lambda=1\), one has \(|T_\lambda K|=|K|=1\). Moreover,
\(T_\lambda\) sends the segment \([0,d]e_1\) onto
\([0,\lambda d]e_1\), dilates the direction \(e_1\) by the factor \(\lambda\),
and contracts every direction orthogonal to \(e_1\) by the factor
\(\lambda^{-1/(m-1)}\). Hence \(D(T_\lambda K)=\lambda d\).

Finally, since \(K\subset [0,d]\times B_{m-1}(0,R\cdot d^{-1/(m-1)})\), we have the inclusion
\[
T_\lambda K
\subset
[0,\lambda d]\times
B_{m-1}\left(0,R(\lambda d)^{-1/(m-1)}\right),
\]
which yields \(T_\lambda K\in\mathcal C_R(\lambda d)\).

%__%\proofstep{Diameter control along the path induced by Minkowski interpolation.}{step:diameter-control-cylindrical-classes}

\textbf{Step 5. Diameter control along the path induced by Minkowski interpolation.}

Let \(R>0\), \(d>0\), and \(K,L\in\mathcal C_R(d)\). For \(t\in[0,1]\), set
$M_t:=(1-t)K+tL$ and $
\widetilde M_t:=|M_t|^{-1/m}M_t$. Since \(0,de_1\in K\cap L\), and since the diameter is subadditive with respect
to Minkowski addition, one has \(D(M_t)=d\) for every \(t\in[0,1]\).

Moreover, since both \(K\) and \(L\) are contained in the cylinder
$
[0,d]\times B_{m-1}\left(0,Rd^{-1/(m-1)}\right),$
the same holds for the interpolation \(M_t\). Consequently,
$
|M_t|\le \omega_{m-1}R^{m-1}$. After normalization by volume, we can write
\[
D(\widetilde M_t)
=
|M_t|^{-1/m}D(M_t)
=
\frac{d}{|M_t|^{1/m}}\ge \frac{d}{\bigl(\omega_{m-1}R^{m-1}\bigr)^{1/m}}.
\]

On the other hand, the Brunn--Minkowski inequality,
\[
|M_t|^{1/m}
\ge
(1-t)|K|^{1/m}+t|L|^{1/m}
=
1,
\]
implies that \(D(\widetilde M_t)\le d\). 

Therefore, for every \(t\in[0,1]\),
\[
\frac{d}{\bigl(\omega_{m-1}R^{m-1}\bigr)^{1/m}}
\le
D(\widetilde M_t)
\le
d.
\]

In particular, the first coordinate of every point of the path induced by the
normalized Minkowski interpolation belongs to the interval $\left[
\bigl(\omega_{m-1}R^{m-1}\bigr)^{-1/m}d,
d
\right]$.

%__%\proofstep{Conclusion via the topological index.}{step:conclusion-via-topological-index} \label{step:conclusion-via-topological-index}

\textbf{Step 6. Conclusion via the topological index. }

We recall that for each diameter \(d\ge D(B)\), 
\[
f(d):=\inf\{h(K):K\in\mathcal K_1^m,\ D(K)=d\},
\qquad
g(d):=\sup\{h(K):K\in\mathcal K_1^m,\ D(K)=d\},
\]
and set
$E:=\{(x,y)\in\mathbb R^2:\ x\ge D(B),\ f(x)\le y\le g(x)\}$.

We have already shown that \(\mathcal D_m\) is a closed subset of \(E\). We now prove that \(\mathcal D_m=E\). This will imply that \(\mathcal D_m\) is simply connected.
Suppose, for contradiction, that there exists \(z_0=(x_0,y_0)\in E\setminus\mathcal D_m\). Since the extrema defining \(f\) and \(g\) are attained, the boundary points \((x,f(x))\) and \((x,g(x))\) belong to \(\mathcal D_m\). By the continuity of \(f\) and \(g\), it follows that \(z_0\) is an interior point of \(E\). Since \(\mathcal D_m\) is closed, there exists \(r>0\) such that \(B(z_0,r)\subset E\setminus\mathcal D_m\).

Set \(\widetilde{d_0}:=x_0+\frac r2\).
Since \(B(z_0,r)\subset E\), one has \((\widetilde{d_0},y_0)\in E\), and therefore
$f(\widetilde{d_0})\le y_0\le g(\widetilde{d_0})$. We then choose \(L_0,K_0\in\mathcal K_1^m\) such that $D(L_0)=D(K_0)=\widetilde{d_0},\  h(L_0)=f(\widetilde{d_0})\ \text{and}\ h(K_0)=g(\widetilde{d_0})$. After applying suitable rigid motions, we may assume that chosen diametral segments of \(L_0\) and \(K_0\) both coincide with \([0,\widetilde d_0]e_1\). Choose \(R_{L_0},R_{K_0}>0\) such that
\[
L_0,K_0\subset [0,\widetilde d_0]\times B_{m-1}\left(0,R_{L_0}\widetilde {d_0}^{-1/(m-1)}\right).
\]

Take $R_0:=\max\{R_{L_0},R_{K_0}\}$. We then have
$L_0,K_0\in\mathcal C_{R_0}(\widetilde d_0)$. Also, if \(L,K\in\mathcal C_{R_0}(d)\), we denote by \(\sigma_{L,K}\) the associated closed curve.

Let $A_{R_0}:=\bigl(\omega_{m-1}R_0^{m-1}\bigr)^{1/m}$. By \textbf{Step 5},
every point of \(\sigma_{L,K}\) has first coordinate in the interval
$[d/A_{R_0},\,d]$. In particular, if \(d\) is sufficiently large, then \(\sigma_{L,K}\) lies entirely in the half-plane \(x>x_0+r\). Hence
\(
\operatorname{ind}(\sigma_{L,K},z_0)=0.
\), and therefore the topological index (or winding number) of \(\sigma_{L,M}\) with respect to \(z_0\) is zero, that is,
$\operatorname{ind}(\sigma_{L,M},z_0)=0$.

We now define
\[
I_{R_0}:=
\Bigl\{
d\ge \widetilde d_0:\ \exists\,L,K\in\mathcal C_{R_0}(d)
\text{ with } h(L)\le h(K)
\text{ and } \operatorname{ind}(\sigma_{L,K},z_0)\neq 0
\Bigr\}.
\]

The set \(I_{R_0}\) is nonempty. Indeed, for \(d=\widetilde d_0\), the auxiliary vertical segment in the definition of \(\sigma_{L_0,K_0}\) contains the point \((\widetilde d_0,y_0)\in B(z_0,r)\). On the other hand, the non-vertical part of the curve, namely the image of \(t\in[0,1]\mapsto (D(\widetilde M_t),h(\widetilde M_t))\), is contained in the Blaschke--Santal\'o diagram \(\mathcal D_m\), and is therefore disjoint from \(B(z_0,r)\). Moreover, this part of the curve is contained in the strip
$\{(x,y);\ {d_0}/{A_{R_0}}\le x\le \widetilde d_0\}$. Thus, the curve cannot close by passing through the right-hand side of the disk, and the resulting closed curve has nonzero topological index with respect to \(z_0\). Hence
\(
\operatorname{ind}(\sigma_{L_0,K_0},z_0)\neq 0.
\)
Therefore \(\widetilde d_0\in I_{R_0}\), and so \(I_{R_0}\neq\varnothing\). On the other hand, by the diameter bound established above, \(I_{R_0}\) is bounded from above. We may therefore set $d^{*}:=\sup I_{R_0}$.

We now distinguish two cases.

\medskip

\noindent\textbf{Case 1: \(d^{*}\notin I_{R_0}\).}

In this case, there exists a sequence \(d_n\nearrow d^{*}\) with \(d_n\in I_{R_0}\), together with two sequences
\(
L_n,K_n\in\mathcal C_{R_0}(d_n)
\)
such that
\(
z_0\in\operatorname{int}(\sigma_{L_n,K_n}).
\)

Since the diameters \(d_n\) are bounded above by \(d^{*}\), the Blaschke selection theorem yields, after extraction of a common subsequence still denoted by \(L_n\) and \(K_n\), such that 
\(
L_n\underset{n\rightarrow+\infty}{\longrightarrow} L,
\;\text{and}\  
K_n\underset{n\rightarrow+\infty}{\longrightarrow} K
\)
for some \(L,K\in\mathcal K_1^m\). 

By the closedness property proved in \textbf{Step 3}, and since the diameter is continuous under Hausdorff convergence, it follows that
$L,K\in\mathcal C_{R_0}(d^{*})$.
Hence, by the uniform convergence theorem established earlier, the associated curves \(\sigma_{L_n,K_n}\) converge uniformly to \(\sigma_{L,K}\).

Since \(z_0\notin \operatorname{int}(\sigma_{L,K})\), the continuity of the topological index under uniform convergence of closed curves away from the base point implies that
\(
\operatorname{ind}(\sigma_{L_n,K_n},z_0)=\operatorname{ind}(\sigma_{L,K},z_0)
\)
for all \(n\). 

On the other hand, since \(z_0\in\operatorname{int}(\sigma_{L_n,K_n})\), one has
\(
\operatorname{ind}(\sigma_{L_n,K_n},z_0)\neq 0,
\)
whereas \(d^{*}\notin I_{R_0}\) implies
\(
\operatorname{ind}(\sigma_{L,K},z_0)=0.
\)
This is a contradiction.

\medskip

\noindent\textbf{Case 2: \(d^{*}\in I_{R_0}\).}

By the definition of \(I_{R_0}\), there exist
\(
L,K\in\mathcal C_{R_0}(d^{*})
\)
such that
\(
z_0\in\operatorname{int}(\sigma_{L,K}).
\)

We now apply the affine perturbation argument. After suitable rigid motions, we may assume that the chosen diameter segments of \(L\) and \(K\) both coincide with the segment \([0,d^{*}]e_1\). For \(n\ge1\), let \(T_n\in SL(m)\) be defined by
\[
(T_n)_{11}=\frac{n+1}{n},
\qquad
(T_n)_{ii}=\Bigl(\frac{n}{n+1}\Bigr)^{1/(m-1)}\quad\text{for }i\ge2,
\qquad
(T_n)_{ij}=0\quad\text{for }i\neq j,
\]
and set
$L_n:=T_n(L)$ and $K_n:=T_n(K)$. Since \(T_n\underset{n\rightarrow+\infty}{\longrightarrow} I\), we have
$L_n\underset{n\rightarrow+\infty}{\longrightarrow} L$ and $K_n\underset{n\rightarrow+\infty}{\longrightarrow} K$
in the Hausdorff metric. Moreover, \(\det T_n=1\), hence
\(
|L_n|=|K_n|=1.
\)

Set $d_n:=\Bigl(1+\frac1n\Bigr)d^{*}$.
Since \(T_n\) stretches the in the direction of \(e_1\) by the factor \((n+1)/n\) and contracts the remaining orthogonal directions, we have
$D(L_n)=D(K_n)=d_n>d^{*}$.

By \textbf{Step 4},
$L_n,K_n\in\mathcal C_{R_0}(d_n)$. 
Again, by the uniform convergence of the associated curves, \(\sigma_{L_n,K_n}\) converge uniformly to \(\sigma_{L,K}\) on \([0,2]\). the continuity of the topological index yields, for all sufficiently large \(n\),
$\operatorname{ind}(\sigma_{L_n,K_n},z_0)=\operatorname{ind}(\sigma_{L,K},z_0)$.
Whereas,
$\operatorname{ind}(\sigma_{L_n,K_n},z_0)= 0
\ \text{and}\ 
\operatorname{ind}(\sigma_{L,K},z_0)\neq0$,
which is a contradiction.

The contradiction in both cases (\textbf{Case 1} and \textbf{Case 2}) shows that \(E\setminus\mathcal D_m=\varnothing\). Hence \(\mathcal D_m=E\). Since \(E\) is the region between the graphs of the continuous functions \(f\) and \(g\) on \([D(B),+\infty)\), it is simply connected. 
\end{proof}

The representation of \(\mathcal D_m\) as the region between the graphs of the functions
\(f\) and \(g\) has the following immediate consequence:
\begin{corollary}
The diagram \(\mathcal D_m\) is vertically convex. If \(m=2\), then the diagram
\(\mathcal D_2\) is also horizontally convex.
\end{corollary}
\section*{Acknowledgment}
The author was supported by the French National Research Agency under the France~2030 program,
grant ANR-23-EXES-0005 (project Gardener/MASRA), and by the Occitanie Region. He would like to thank Ilias Ftouhi for useful discussions.
\bibliographystyle{plain}
\bibliography{references}

\end{document}